\documentclass{siamart190516}

\usepackage{lipsum}
\usepackage{amsfonts}
\usepackage{graphicx}
\usepackage{epstopdf}
\usepackage{algorithmic}
\ifpdf
  \DeclareGraphicsExtensions{.eps,.pdf,.png,.jpg}
\else
  \DeclareGraphicsExtensions{.eps}
\fi

\newsiamremark{remark}{Remark}
\newsiamremark{hypothesis}{Hypothesis}
\crefname{hypothesis}{Hypothesis}{Hypotheses}
\newsiamthm{claim}{Claim}

\headers{Finite elements for symmetric tensors}{L. Chen and X. Huang}

\title{Finite elements for \revision{div and divdiv conforming} symmetric tensors in arbitrary dimension\thanks{Submitted to the editors \today.
\funding{The first author was supported by NSF DMS-1913080 and DMS-2012465. The second author was supported by the National Natural Science Foundation of China Projects 12171300 and 11771338, the Natural Science Foundation of Shanghai 21ZR1480500, and the Fundamental Research Funds for the Central
Universities 2019110066.}}}

\author{Long Chen\thanks{Department of Mathematics, University of California at Irvine, Irvine, CA 92697, USA 
  (\email{chenlong@math.uci.edu}).}
\and 
Xuehai Huang\thanks{Corresponding author. School of Mathematics, Shanghai University of Finance and Economics, Shanghai 200433, China 
  (\email{huang.xuehai@sufe.edu.cn}).}
}

\usepackage{amsopn}

\ifpdf
\hypersetup{
  pdftitle={Finite elements for symmetric tensors},
  pdfauthor={L. Chen and X. Huang}
}
\fi

\usepackage{amssymb}
\usepackage[leqno]{amsmath}
\usepackage{mathrsfs}
\usepackage{stmaryrd}
\usepackage{chemarrow}
\usepackage{enumerate}
\usepackage{graphicx}
\usepackage{subfigure}  
\usepackage{hyperref}
\usepackage[hyperpageref]{backref}
\usepackage[all]{xy}
\usepackage{tikz}
\usetikzlibrary{arrows}

\usepackage{multirow}

\newcommand{\revision}[1]{\textcolor{black}{#1}}

  \newcounter{mnote}
  \let\oldmarginpar\marginpar
    \renewcommand\marginpar[1]{\-\oldmarginpar[\raggedleft\footnotesize #1]%
    {\raggedright\footnotesize #1}}

\newcommand{\dx}{\,{\rm d}x}
\newcommand{\dd}{\,{\rm d}}
\newcommand{\bs}{\boldsymbol}

\DeclareMathOperator*{\img}{img}

\newcommand{\curl}{\operatorname{curl}}
\renewcommand{\div}{\operatorname{div}}
\newcommand{\grad}{\operatorname{grad}}
\DeclareMathOperator*{\tr}{tr}

\newcommand{\sym}{\operatorname{sym}}
\newcommand{\skw}{\operatorname{skw}}

\newcommand{\defm}{\operatorname{def}}

\numberwithin{equation}{section}

\begin{document}

\maketitle

\begin{abstract}
Several div-conforming and divdiv-conforming finite elements for symmetric tensors on simplexes in arbitrary dimension are constructed in this work. The shape function space is first split as the trace space and the bubble space. The later is further decomposed into the null space of the differential operator and its orthogonal complement. Instead of characterization of these subspaces of the shape function space, characterization of the dual spaces are provided. Vector div-conforming finite elements are firstly constructed as an introductory example. Then new symmetric div-conforming finite elements are constructed. The dual subspaces are then used as build blocks to construct divdiv conforming finite elements.
\end{abstract}

\begin{keywords}
symmetric tensor, div-conforming finite elements, divdiv-conforming finite elements, space decomposition, dual approach
\end{keywords}

\begin{AMS}
  65N30, 74S05
\end{AMS}

\section{Introduction}
In this paper we construct div-conforming finite elements and divdiv-conforming finite elements for symmetric tensors on simplexes in arbitrary dimension. A finite element on a geometric domain $K$ is defined as a triple $(K, V, {\rm DoF})$ by Ciarlet in~\cite{Ciarlet1978}, where $V$ is the finite-dimensional space of shape functions and the set of degrees of freedom (DoFs) is a basis of the dual space $V'$ of $V$. The shape functions are usually polynomials. The key is to identify an appropriate basis of $V'$ to enforce the continuity of the functions across the boundary of the elements so that the global finite element space is a subspace of some Sobolev space $H(\dd,\Omega)$, where $\Omega\subset \mathbb R^d$ is a domain and $\dd$ is a generic differential operator. 
%
%

\begin{figure}[htbp]
\begin{center}
\includegraphics[width=5.4cm]{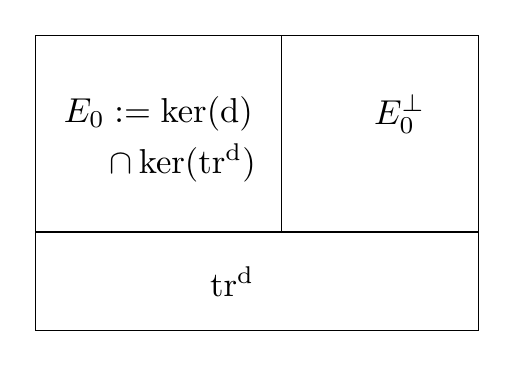}
\caption{Decomposition of a generic finite element space\revision{.}}
\label{fig:femdec}
\end{center}
\end{figure}

Denote by ${\rm tr}^{\rm d}$ the trace operator associated to ${\rm d}$ and the bubble function space $\mathbb B({\rm d}):=\ker({\rm tr}^{\rm d})\cap V$. We shall decompose \revision{$V =  \mathbb B({\rm d}) \oplus \mathcal E({\rm img}({\rm tr}^{\rm d}))$, where $\mathcal E$ is an injective extension operator ${\rm img}({\rm tr}^{\rm d}) \to V$,} and find degrees of freedom of each subspace by:
\begin{enumerate}
 \item characterization of $({\rm img}({\rm tr}^{\rm d}))'$ using the Green's formula; 
 
 \item characterization of $\mathbb B'({\rm d})$ through the polynomial complexes.
\end{enumerate}
In the characterization of $\mathbb B'({\rm d})$, we will use the differential operator $\dd$ to further split $\mathbb B({\rm d})$ into two subspaces $E_0:=\mathbb B({\rm d})\cap\ker({\rm d})$ and $E_0^{\bot}:=\mathbb B({\rm d})/ E_0$ (see Figure~\ref{fig:femdec}). 
\begin{itemize}
\item  A basis of $(E_0^{\bot})'$ is given by $\{(\dd\cdot, p), \; p\in\dd\mathbb B({\rm d}) = \dd E_0^{\bot}\}$; 
\item On the other part $E_0'$, we have two approaches:
\begin{itemize}
 \item the primary approach: $E_0$ is the image of the previous bubble space,%
 \item the dual approach: $E_0'$ is the null space of a Koszul operator.
\end{itemize}
\end{itemize}
The dual approach is simpler and more general. For example, for the elasticity complex, the previous symmetric tensor space is related to the second order differential operator ${\rm inc}$~\cite{ArnoldAwanouWinther2008}. While in the dual approach, we prove that a basis of $E_0'$ is given by $\mathcal N(\ker(\cdot\boldsymbol x)\cap \mathbb P_{k-2}(K;\mathbb S))$. Here to simplify notation, we introduce operator $\mathcal N: U\to V'$ as $\mathcal N(p): = (\cdot, p)$ with $U\subseteq V$ \revision{and $(\cdot,\cdot)$ is the inner product of space $V$ which is usually $L^2$-inner product}. \revision{Generalization of ${\rm inc}$ and its bubble function space to $\mathbb R^d$ is unclear while $E_0' = \mathcal N(\ker(\cdot\boldsymbol x)\cap \mathbb P_{k-2}(K;\mathbb S))$ holds in arbitrary dimension.} 
\revision{Such decomposition $V' \cong  E_0' \oplus (E_0^{\bot})' \oplus ({\rm img}({\rm tr}^{\rm d}))'$ also allows us to construct a new family of $H(\div; \mathbb S)$-conforming finite elements.}



\revision{To show the main idea with easy examples,} we first review the construction of the Brezzi-Douglas-Marini (BDM) element~\cite{BrezziDouglasMarini1986,BrezziDouglasDuranFortin1987} and Raviart-Thomas (RT) element~\cite{RaviartThomas1977,Nedelec1980} for $H(\div)$-conforming elements. The trace space ${\rm tr}^{\div}(\mathbb P_k(K; \mathbb R^d))= \prod_{F\in \partial K} \mathbb P_{k}(F)$. 
By the aid of the space decomposition
$\mathbb P_{k-1}(K; \mathbb R^d) = \grad \mathbb P_{k}(K) \oplus\ker(\cdot\boldsymbol x)\cap \mathbb P_{k-1}(K;\mathbb R^d) $ derived from the dual complex,
we can show $E_0' = \mathcal N( \ker(\cdot\boldsymbol x)\cap \mathbb P_{k-1}(K;\mathbb R^d))$.
For BDM element, the shape function space is $\mathbb P_k(K;\mathbb R^d)$. 
The shape function space for RT element is $\mathbb P_{k+1}^-(K;\mathbb R^d):= \mathbb P_{k}(K;\mathbb R^d)\oplus\mathbb H_{k}(K)\boldsymbol x.$
BDM and RT elements will share the same trace space and $E_0$, while
$$(E_0^{\bot})' = 
\begin{cases}
\mathcal N(\grad \mathbb P_{k-1}(K)) & \text{for BDM element},\\
\mathcal N(\grad \mathbb P_{k}(K))  & \text{for RT element}.
\end{cases}
$$
The dual space $\mathbb B_k'(\div, K) \cong (E_0^{\bot})' \oplus E_0'$ for BDM element can be further merged as 
$$
\mathbb B_k'(\div, K) = \mathcal N\big({\rm ND}_{k-2}(K)\big):=\mathcal N\big(\mathbb P_{k-2}(K;\mathbb R^d) \oplus \mathbb H_{k-2}(K; \mathbb K)\boldsymbol x\big).
$$
We summarize DoFs for BDM element as follows
\begin{align}
(\boldsymbol  v\cdot\boldsymbol n, q)_F& \quad\forall~q\in\mathbb P_{k}(F) \textrm{ for each } F\in \partial K, \label{intro:BDMdof1}\\
(\boldsymbol v, \boldsymbol q)_K &\quad\forall~\boldsymbol q\in{\rm ND}_{k-2}(K), \label{intro:BDMdof2}
\end{align}
and the interior moments \eqref{intro:BDMdof2} can be further split as
\begin{align}
(\boldsymbol v, \boldsymbol q)_K &\quad\forall~\boldsymbol q\in\grad \mathbb P_{k-1}(K), \label{intro:BDMdof3} \\
(\boldsymbol v, \boldsymbol q)_K &\quad\forall~\boldsymbol q\in \ker(\cdot\boldsymbol x)\cap \mathbb P_{k-1}(K;\mathbb R^d). \label{intro:BDMdof4}
\end{align}
Enriching \eqref{intro:BDMdof3} to $\mathcal N(\grad \mathbb P_{k}(K))$, we then get RT elements.
 
We then apply our approach for a more challenging problem: div-conforming finite elements for symmetric tensors, which is used in the mixed finite element methods for the stress-displacement formulation of the elasticity system. 
Several div-conforming finite elements for symmetric tensors were designed in \cite{ArnoldWinther2002,Adams;Cockburn:2005Finite,ArnoldAwanouWinther2008,HuZhang2016,Hu2015a,HuZhang2015} on simplices, 
but our elements are new and construction is more systematical. \revision{Let $\Pi_F\bs \tau$ be the projection of column vectors of $\bs \tau$ to the plane $F$.}
The space of shape functions is $\mathbb P_k(K;\mathbb S)$, and DoFs are
\begin{align}
\boldsymbol \tau (\delta) & \quad\forall~\delta\in \mathcal V(K), \label{intro:HdivSBDMfemdof1}\\
(\boldsymbol  n_i^{\intercal}\boldsymbol \tau\boldsymbol n_j, q)_F & \quad\forall~q\in\mathbb P_{k+r-d-1}(F),  F\in\mathcal F^r(K),\;  \label{intro:HdivSBDMfemdof2}\\
&\quad\quad i,j=1,\cdots, r, \textrm{ and } r=1,\cdots, d-1, \notag\\
(\Pi_F\boldsymbol \tau\boldsymbol n, \boldsymbol q)_F & \quad\forall~\boldsymbol q\in {\rm ND}_{k-2}(F),  F\in\mathcal F^1(K),\label{intro:HdivSBDMfemdof3} \\
\label{intro:HdivSBDMfemdof4}(\boldsymbol \tau, \boldsymbol q)_K &\quad \forall~\boldsymbol q\in \mathbb P_{k-2}(K;\mathbb S).
\end{align}
The symmetry of the shape function and the trace $\boldsymbol\tau\boldsymbol n$ on $(d-1)$-dimensional faces leads to the degrees of freedom \eqref{intro:HdivSBDMfemdof1}-\eqref{intro:HdivSBDMfemdof2}, which will determine the normal-normal component $\boldsymbol n^{\intercal}\boldsymbol \tau \boldsymbol n$.
The set of degrees of freedom \eqref{intro:HdivSBDMfemdof3} is for the face bubble part of the tangential-normal component $\Pi_F\boldsymbol \tau\boldsymbol n$, cf. \eqref{intro:BDMdof2}, which 
differs from that of Hu's element in~\cite{Hu2015a} for $d\geq3$. 
The bubble function space $\mathbb B_k(\div, K;\mathbb S)$ can be decomposed into two parts $E_{0,k}(\mathbb S) := \mathbb B_k(\div, K; \mathbb S)\cap \ker(\div)$ and $E_{0,k}^{\bot}(\mathbb S) := \mathbb B_k(\div, K; \mathbb S)/ E_{0,k}(\mathbb S)$. 
We show that
\begin{equation}\label{intro:E0S}
(E_{0,k}^{\bot}(\mathbb S))' = \mathcal N(\defm \mathbb P_{k-1}(K,\mathbb R^d)),\quad
E_{0,k}'(\mathbb S) = \mathcal N(\ker (\cdot\boldsymbol x )\cap \mathbb P_{k-2}(K;\mathbb S)).
\end{equation}
\revision{A new family of $H(\div; \mathbb S)$-conforming elements  is devised with the shape function space $\mathbb P_{k+1}^{-}(K;\mathbb S) :=\mathbb P_k(K;\mathbb S)+ E_{0,k+1}^{\bot}(\mathbb S)$, and enrich DoF $(E_{0,k}^{\bot}(\mathbb S))'$ to $(E_{0,k+1}^{\bot}(\mathbb S))'$ so that $\div\mathbb P_{k+1}^{-}(K;\mathbb S) = \mathbb P_{k}(K;\mathbb R^d)$.}
%

Motivated by the recent construction \cite{Hu;Ma;Zhang:2020family} in two and three dimensions, the previous div-conforming finite elements for symmetric tensors are then revised to acquire $H(\div\div)\cap H(\div)$-conforming finite elements for symmetric tensors in arbitrary dimension. Using the building blocks in the BDM element and div-conforming $\mathbb P_k(K;\mathbb S)$ element, we construct the following DoFs
\begin{align}
\boldsymbol \tau (\delta) & \quad\forall~\delta\in \mathcal V(K), \label{intro:HdivdivSfemdof1}\\
(\boldsymbol  n_i^{\intercal}\boldsymbol \tau\boldsymbol n_j, q)_F & \quad\forall~q\in\mathbb P_{k+r-d-1}(F),  F\in\mathcal F^r(K),\;  \label{intro:HdivdivSfemdof2}\\
&\quad\quad i,j=1,\cdots, r, \textrm{ and } r=1,\cdots, d-1, \notag\\
(\Pi_F\boldsymbol \tau\boldsymbol n, \boldsymbol q)_F & \quad\forall~\boldsymbol q\in {\rm ND}_{k-2}(F),  F\in\mathcal F^{1}(K),\label{intro:HdivdivSfemdof3}\\
(\boldsymbol  n^{\intercal}\div \boldsymbol \tau, p)_F &\quad\forall~p\in\mathbb P_{k-1}(F), F\in \mathcal F^{1}(K),\label{intro:HdivdivSfemdof4}\\
(\boldsymbol \tau, \defm\boldsymbol q)_K &\quad \forall~\boldsymbol q\in {\rm ND}_{k-3}(K),\label{intro:HdivdivSfemdof5}\\
(\boldsymbol \tau, \boldsymbol q)_K &\quad \forall~\boldsymbol q\in  \ker (\cdot\boldsymbol x )\cap \mathbb P_{k-2}(K;\mathbb S) \label{intro:HdivdivSfemdof6}.
\end{align}
The degree of freedom \eqref{intro:HdivdivSfemdof4} is to enforce $\div\boldsymbol\tau$ is $H(\div)$-conforming and thus $\bs \tau \in H(\div\div)\cap H(\div)$. \revision{DoF \eqref{intro:HdivdivSfemdof6} is $E_{0,k}'(\mathbb S)$ shown in \eqref{intro:E0S} and \eqref{intro:HdivdivSfemdof4}-\eqref{intro:HdivdivSfemdof5} is a further decomposition of $\div E_{0,k}^{\bot}(\mathbb S)$ by the trace-bubble decomposition cf. \eqref{intro:BDMdof1}-\eqref{intro:BDMdof2}. }

We then modify this element slightly to get $H(\div\div)$-conforming symmetric finite elements generalizing the divdiv-conforming element in two and three dimensions~\cite{Chen;Huang:2020Finite,ChenHuang2020}. The degrees of freedom are given by
\begin{align}
\boldsymbol \tau (\delta) & \quad\forall~\delta\in \mathcal V(K), \label{intro:newHdivdivSfemdof1}\\
(\boldsymbol  n_i^{\intercal}\boldsymbol \tau\boldsymbol n_j, q)_F & \quad\forall~q\in\mathbb P_{k+r-d-1}(F),  F\in\mathcal F^r(K),\;  \label{intro:newHdivdivSfemdof2}\\
&\quad\quad i,j=1,\cdots, r, \textrm{ and } r=1,\cdots, d-1, \notag\\
(\Pi_F\boldsymbol \tau\boldsymbol n, \boldsymbol q)_F & \quad\forall~\boldsymbol q\in {\rm ND}_{k-2}(F),  F\in\mathcal F^{1}(K),\label{intro:newHdivdivSfemdof4}\\
(\boldsymbol n^{\intercal}\div \boldsymbol \tau +  \div_F(\boldsymbol\tau \boldsymbol n), p)_F &\quad\forall~p\in\mathbb P_{k-1}(F), F\in \mathcal F^{1}(K),\label{intro:newHdivdivSfemdof3}\\
(\boldsymbol \tau, \defm\boldsymbol q)_K &\quad \forall~\boldsymbol q\in {\rm ND}_{k-3}(K),\label{intro:newHdivdivSfemdof5}\\
(\boldsymbol \tau, \boldsymbol q)_K &\quad \forall~\boldsymbol q\in  \ker (\cdot\boldsymbol x)\cap \mathbb P_{k-2}(K;\mathbb S) \label{intro:newHdivdivSfemdof6}.
\end{align}
As we mentioned before, \eqref{intro:newHdivdivSfemdof1}-\eqref{intro:newHdivdivSfemdof4} will determine the trace $\boldsymbol \tau\boldsymbol n$, and consequently $\div_F(\boldsymbol\tau \boldsymbol n)$. 
The only difference is that \eqref{intro:HdivdivSfemdof4} is replaced by \eqref{intro:newHdivdivSfemdof3}, 
which agrees with a trace operator of $\div\div$-operator. Such modification is from the requirement of $H(\div\div)$-conformity: $\boldsymbol n^{\intercal}\boldsymbol\tau\boldsymbol n$ and $\boldsymbol n^{\intercal}\div \boldsymbol \tau +  \div_F(\boldsymbol\tau \boldsymbol n)$ are continuous. 
Therefore \eqref{intro:newHdivdivSfemdof4} for $\Pi_F\boldsymbol \tau\boldsymbol n$ is considered as interior to $K$, i.e., it is not single-valued across simplices. 

\revision{In our recent work \cite{ChenHuang2020,Chen;Huang:2020Finite}, we have constructed $H(\div\div)$-conforming symmetric finite elements for $d=2,3$. The dual space $(\tr^{\div\div} (\mathbb P_k(K;\mathbb S)))'$ is given by DoFs \eqref{intro:newHdivdivSfemdof1}-\eqref{intro:newHdivdivSfemdof3} but without \eqref{intro:newHdivdivSfemdof4} as $\Pi_F\boldsymbol \tau\boldsymbol n$ is not part of the trace of $\div\div$ operator. Let $E_0(\div\div,\mathbb S) := \mathbb B_k(\div\div, K; \mathbb S)\cap \ker(\div\div)$ and $E_0^{\bot}(\div\div,\mathbb S) := \mathbb B_k(\div\div, K; \mathbb S)/ E_0(\div\div,\mathbb S)$. Then $(E_0^{\bot}(\div\div,\mathbb S))' = \mathcal N(\nabla^2\mathbb P_{k-2}(K))$ but the identification of $E_0'(\div\div,\mathbb S)$ is very tricky in three dimensions, We use the primary approach to get $E_0'(\div\div,\mathbb S) = \mathcal N(\sym \curl \mathbb B_{k+1}(\sym\curl, K;\mathbb T))$. Characterization of $\mathbb B_{k+1}(\sym\curl, K;\mathbb T)$ is hard to generalize to arbitrary dimension. When using the dual approach, it turns out $\mathcal N(\ker(\boldsymbol x^{\intercal}\cdot\boldsymbol x)\cap \mathbb P_{k-1}(K;\mathbb S))$ is a strict subspace of $ E_0'(\div\div,\mathbb S)$ as the dimension cannot match. An extra DoF on one face $(\boldsymbol \tau\boldsymbol n, \boldsymbol n\times\boldsymbol xq)_{F_1}, ~q\in\mathbb P_{k-2}(F_1)$ is introduced to fill the gap. Such fix in three dimensions seems not easy to be generalized to arbitrary dimension. } 

\revision{In \eqref{intro:newHdivdivSfemdof5}, if we further decompose ${\rm ND}_{k-3}(K)=\grad\mathbb P_{k-2}(K) \oplus \mathbb P_{k-3}(K;\mathbb K)\bs x$, based on our new element, we can obtain another characterization of  $ E_0'(\div\div,\mathbb S)$ as the sum of DoFs \eqref{intro:newHdivdivSfemdof4}, \eqref{intro:newHdivdivSfemdof6}, and $\mathcal N(\defm \mathbb P_{k-3}(K;\mathbb K)\bs x)$. 
}

Furthermore, a new family of $\mathbb P_{k+1}^{-}(K;\mathbb S)$ type $H(\div\div)\cap H(\div)$-conforming  and $H(\div\div)$-conforming finite elements are developed. The shape function space is enriched to $\mathbb P_{k+1}^{-}(K;\mathbb S):=\mathbb P_k(K;\mathbb S) \oplus \boldsymbol x\boldsymbol x^{\intercal}\mathbb H_{k-1}(K)$. The range $\div\div \mathbb P_{k+1}^{-}(K;\mathbb S)$ is enriched to $\mathbb P_{k-1}(K)$ and so is \revision{$(E_0^{\bot}(\div\div,\mathbb S))' = \mathcal N(\nabla^2\mathbb P_{k-1}(K))$}. But the trace DoFs and $E_0'(\div\div,\mathbb S)$ are unchanged. Such $\mathbb P_{k+1}^{-}(K;\mathbb S)$ type $\div\div$-conforming elements for symmetric tensors are new and not easy to construct without exploring the decomposition of the dual spaces.

The rest of this paper is organized as follows. Preliminaries are given in Section~\ref{sec:pre}. We review the construction of div-conforming elements in Section~\ref{sec:divvec}. In Section~\ref{sec:divS}, new div-conforming elements for symmetric tensors are designed. And we consider the construction of divdiv-conforming elements in Section~\ref{sec:divdivS}.

\section{Preliminary}\label{sec:pre}

\subsection{Notation}
Let $K\subset\mathbb R^d$ be a simplex. For $r=1, 2, \cdots, d$, denote by $\mathcal{F}^r(K)$ the set of all $(d-r)$-dimensional faces
of $K$.
The \revision{superscript} $r$ in $\mathcal{F}^r(K)$ represents the co-dimension of a $(d-r)$-dimensional face $F$. 
Set $\mathcal V(K):=\mathcal{F}^d(K)$ as the set of vertices.
Similarly, for $F\in\mathcal{F}^r(K)$, define
\[
\mathcal F^1(F):=\{e\in\mathcal{F}^{r+1}(K): e\subset\partial F\}.
\]
For any $F\in\mathcal{F}^r(K)$ with $1\leq r\leq d-1$, let $\boldsymbol n_{F,1}, \cdots, \boldsymbol n_{F,r}$ be its
mutually perpendicular unit normal vectors, and $\boldsymbol t_{F,1}, \cdots, \boldsymbol t_{F,d-r}$ be its
mutually perpendicular unit tangential vectors.
We abbreviate $\boldsymbol n_{F,1}$ as $\boldsymbol n_{F}$ or $\bs n$ when $r=1$.
We also abbreviate $\boldsymbol n_{F,i}$ and $\boldsymbol t_{F,i}$ as $\boldsymbol n_{i}$ and $\boldsymbol t_{i}$ respectively if not causing any confusion. 
For any $F\in\mathcal{F}^1(K)$ and $e\in\mathcal{F}^1(F)$, denote by $\boldsymbol n_{F,e}$ the unit outward normal to $\partial F$ being parallel to $F$.

Given a face $F\in \mathcal F^1(K)$, and a vector $\boldsymbol v\in \mathbb R^d$, define 
$$
\Pi_F\boldsymbol v= (\boldsymbol n_F\times \boldsymbol v)\times \boldsymbol n_F = (\boldsymbol I - \boldsymbol n_F\boldsymbol n_F^{\intercal})\boldsymbol v
$$ 
as the projection of $\boldsymbol v$ onto the face $F$. \revision{For a matrix $\bs \tau\in \mathbb R^{d\times d}$, $\Pi_F\bs \tau$ is applied to each column vector of $\bs \tau$.}
Given a scalar function $v$, define the surface gradient on $F\in \mathcal F^r(K)$ as
\begin{equation*}
\nabla_{F}v:=\Pi_F\nabla v = \nabla v-\sum_{i=1}^r\frac{\partial v}{\partial n_{F,i}}\boldsymbol n_{F,i}=\sum_{i=1}^{d-r}\frac{\partial v}{\partial t_{F,i}}\boldsymbol t_{F,i},
\end{equation*}
namely the projection of $\nabla v$ to the face $F$, which is independent of the choice of the normal vectors. Denote by $\div_{F}$ the corresponding surface divergence.

\subsection{Polynomial spaces}
We recall some results about polynomial spaces on a bounded and topologically trivial domain $D\subset \mathbb R^d$.
Without loss of generality, we assume $\boldsymbol 0\in D$.
Given 
a non-negative integer $k$, 
let $\mathbb P_k(D)$ stand for the set of all polynomials in $D$ with the total degree no more than $k$, and $\mathbb P_k(D; \mathbb{X})$ denote the tensor or vector version.
Let $\mathbb H_k(D):=\mathbb P_k(D)\backslash\mathbb P_{k-1}(D)$ be the space of homogeneous polynomials of degree $k$. 
Recall that 
$$
\dim \mathbb P_{k}(D) = { k + d \choose d} = { k + d \choose k},\quad \dim \mathbb H_{k}(D) = { k + d -1  \choose d-1} = { k + d - 1 \choose k} 
$$ for a $d$-dimensional domain $D$.

By Euler's formula, we have
\begin{align}
\label{eq:homogeneouspolyprop}
\boldsymbol x\cdot\nabla q &= kq\quad\forall~q\in\mathbb H_k(D),\\
\label{eq:Hkdiv}
\div(\boldsymbol x q) &= (k+d)q\quad\forall~q\in\mathbb H_k(D)
\end{align}
for integer $k\geq 0$. 
\subsection{Dual spaces}
Consider a Hilbert space $V$ with the inner product $(\cdot,\cdot)$. 
Let $U\subseteq V$, then define $\mathcal N: U\to V'$ as follows: for any $p\in U$, $\mathcal N(p)\in V'$ is given by
$$
\langle \mathcal N(p), \cdot \rangle = (\cdot, p).
$$

When $V$ is a subspace of an ambient Hilbert space $W$, we use the inclusion $\mathcal I: V\hookrightarrow W$ to denote the embedding of $V$ into $W$. Then the dual operator $\mathcal I': W'\to V'$ is onto. That is for any $N\in W'$, $\mathcal I'N\in V'$ is defined as
$\langle \mathcal I'N, v\rangle = \langle N, \mathcal I v\rangle. $

%

Consider the case \revision{the finite-dimensional sub-space} $V\subseteq W$ and a subspace $P'\subseteq W'$, then  to prove $V' = \mathcal I'P'$, it suffices to show
\begin{equation}\label{eq:dualsubspace}
\text{ for any } v\in V, \; \text{ if } N(v) = 0, \,\text{ for all }~N\in P', \text{ then } v = 0.
\end{equation}
Note that it means $\mathcal I'$ is onto but may not be into. That is $\dim P'$ might be larger than $\dim V'$. \revision{A {\rm DoF} is in general a functional with a domain larger than $V$. It is less rigorous to write $V'\subseteq P'$ as those two subspaces consists of functionals with different domains. The mapping $\mathcal I'$ is introduced as a bridge for comparison.}
When $\mathcal I': P' \to V'$ is a bijection, we shall skip $\mathcal I'$ and simply write as $V' =  P'$. To prove $V' =  P'$, besides \eqref{eq:dualsubspace}, dimension count is applied to verify $\dim V' = \dim P'$. 

The art of designing conforming finite element spaces is indeed identifying  appropriate {\rm DoFs} to enforce the continuity of the function across the boundary of the elements. Take $V = \mathbb P_k(K)$ as an example. A naive choice is $\mathcal N(\mathbb P_k(K)) = V'$ but such basis enforces no continuity on $\partial K$. To be $H^1$-conforming we need a basis for $(\tr(\mathbb P_k(K)))'$ to ensure the continuity of the trace on lower-dimensional faces of an element $K$. Note that as the shape function is a polynomial inside the element, the trace is usually smoother than its Sobolev version, which is known as super-smoothness~\cite{Chui;Lai:1990Multivariate,Lai;Schumaker:2007Spline,Sorokina:2010Intrinsic}. Choice of dual bases is not unique. For example, for $H^1$-conforming finite elements, $V = \mathbb P_k(K)$, Lagrange element and Hermite element will have different bases for $V'$. 

When counting the dimensions, we often use the following simple fact: for a linear operator $T$ defined on a finite dimensional linear space $V$, it holds
\begin{equation*}
\dim V = \dim \ker(T) + \dim \img(T). 
\end{equation*}

\subsection{Simplex and barycentric coordinates}
For $i=1,\cdots, d$, denote by $\boldsymbol e_i\in\mathbb R^d$ the $d$-dimensional vector whose $j$th component is $\delta_{ij}$ for $j=1,\cdots, d$.
Let $K\subset\mathbb R^{d}$ be a non-degenerated simplex with vertices $\boldsymbol x_0, \boldsymbol x_1, \cdots, \boldsymbol x_{d}$. 
Let $F_i\in\mathcal F^1(K)$ be the $(d-1)$-dimensional face opposite to vertex $\boldsymbol x_i$, and $\lambda_i$ be the barycentric coordinate of $\boldsymbol x$ corresponding to vertex $\boldsymbol x_i$, for $i=0, 1, \cdots, d$. Then $\lambda_i(\boldsymbol x)$ is a linear polynomial and $\lambda_i|_{F_i} = 0$. For any sub-simplex $S$ not containing $\boldsymbol x_i$ (and thus $S\subseteq F_i$), $\lambda_i|_S = 0$. On the other hand, for a polynomial $p\in \mathbb P_k(K)$, if $p|_{F_i} = 0$, then $p = \lambda_i q$ for some $q\in \mathbb P_{k-1}(K)$. As $F_i$ is contained in the zero level set of $\lambda_i$, $\nabla \lambda_i$ is orthogonal to $F_i$ and a simple scaling calculation shows the relation $\nabla\lambda_i=-|\nabla\lambda_i|\boldsymbol n_i$, 
where $\boldsymbol n_i$ is the unit outward normal to the face $F_i$ of the simplex $K$ for $i=0, 1, \cdots, d$. 
Clearly $\{ \boldsymbol n_1, \boldsymbol n_2, \cdots, \boldsymbol n_d \}$ spans $\mathbb R^d$. We will identify its dual basis $\{ \boldsymbol l_1, \boldsymbol l_2, \cdots, \boldsymbol l_d \}$, i.e. $(\boldsymbol l_i, \boldsymbol n_j) = \delta_{ij}$ for $i,j=1,2,\cdots, d$. Here the index $0$ is single out for the ease of notation. \revision{We can set an arbitrary vertex as the origin.}

Set $\boldsymbol t_{i,j}:=\boldsymbol x_j-\boldsymbol x_i$ for $0\leq i\neq j\leq d$. By computing the constant directional derivative $\boldsymbol t_{i,j}\cdot\nabla \lambda_{\ell}$ by values on the two vertices, we have
\begin{equation}\label{eq:tijlambdal}
\boldsymbol t_{i,j}\cdot\nabla \lambda_{\ell}=\delta_{j\ell}-\delta_{i\ell}=\begin{cases}
1, & \textrm{ if } \ell=j,\\
-1, & \textrm{ if } \ell=i,\\
0, & \textrm{ if } \ell\neq i,j.
\end{cases}  
\end{equation}
Then it is straightforward to verify $\{\boldsymbol l_i := |\nabla\lambda_{i}| \boldsymbol t_{i,0}\}$ is dual to $\{\boldsymbol n_i\}$. Note that in general neither $\{\boldsymbol n_i\}$ nor $\{\boldsymbol l_i\}$ is an orthonormal basis unless $K$ is a scaling of the reference simplex $\hat K$ with vertices $\boldsymbol 0, \boldsymbol e_1, \cdots, \boldsymbol e_d$. By using the basis $\{\boldsymbol n_i, i=1,2,\cdots, d\}$, we avoid the pull back from the reference simplex. 


Following notation in~\cite{arnold2009geometric},
denote by $\mathbb N^d$ the set of all multi-indices $\alpha=(\alpha_1,\cdots, \alpha_d)$ with integer $\alpha_i\geq0$, and $\mathbb N_0^d$ the set of all multi-indices $\alpha=(\alpha_0,\alpha_1,\cdots, \alpha_d)$ with integer $\alpha_i\geq0$.
For $\boldsymbol x=(x_{1},\cdots, x_{d})$ and $\alpha\in\mathbb N^d$,
define $\boldsymbol x^{\alpha}:=x_{1}^{\alpha_1}\cdots x_{d}^{\alpha_d}$ and $|\alpha|:=\sum\limits_{i=1}^d\alpha_i$. Similarly, for $\boldsymbol \lambda=(\lambda_{0}, \lambda_{1},\cdots, \lambda_{d})$ and $\alpha\in\mathbb N_0^d$,
define $\boldsymbol \lambda^{\alpha}:=\lambda_{0}^{\alpha_0}\lambda_{1}^{\alpha_1}\cdots \lambda_{d}^{\alpha_d}$ and $|\alpha|:=\sum\limits_{i=0}^d\alpha_i$.
The \revision{Bernstein} basis for the space $\mathbb P_k(K)$ consists of all monomials of degree $k$ in the variables $\lambda_i$, i.e., the basis functions are given by
$$
\{\bs \lambda^{\alpha}:=\lambda_{0}^{\alpha_0}\lambda_{1}^{\alpha_1}\cdots\lambda_{d}^{\alpha_d}:\, \alpha\in\mathbb N_0^d,\, |\alpha|=k\}.
$$
Then $\mathbb P_k(K)=\bigg\{\sum\limits_{\alpha\in\mathbb N_0^d, |\alpha|=k}c_{\alpha}\bs \lambda^{\alpha}:\, c_{\alpha}\in\mathbb R \bigg\}$.

\subsection{Tensors}\label{sec:tensor}
Denote by $\mathbb S$ and $\mathbb K$ the subspace of symmetric matrices and \revision{skew}-symmetric matrices of $\mathbb R^{d\times d}$, respectively. 
The set of symmetric tensors $\{\boldsymbol{T}_{i,j}:=\boldsymbol{t}_{i,j}\boldsymbol{t}_{i,j}^{\intercal}\}_{0\leq i<j\leq d}
$ is dual  to $\{\boldsymbol{N}_{i,j}\}_{0\leq i<j\leq d}$, where
$$
\boldsymbol{N}_{i,j}:=\frac{1}{2(\boldsymbol{n}_i^{\intercal}\boldsymbol t_{i,j})(\boldsymbol{n}_j^{\intercal}\boldsymbol t_{i,j})}(\boldsymbol{n}_i\boldsymbol{n}_j^{\intercal} + \boldsymbol{n}_j\boldsymbol{n}_i^{\intercal}).
$$
That is, by direct calculation~\cite[(3.2)]{ChenHuHuang2018},
$$
\boldsymbol{T}_{i,j}:\boldsymbol{N}_{k,\ell}=\delta_{ik}\delta_{j\ell},\quad 0\leq i<j\leq d,\; 0\leq k<\ell\leq d,
$$
where $:$ is the Frobenius inner product of matrices. Assuming $\sum\limits_{0\leq i<j\leq d}c_{ij}\boldsymbol T_{ij} = \boldsymbol0$, then apply the Frobenius inner product with $\boldsymbol N_{k,\ell}$ to conclude $c_{k\ell} = 0$ for all $0\leq k<\ell\leq d$. Therefore both $\{\boldsymbol{T}_{i,j}\}_{0\leq i<j\leq d}$ and $\{\boldsymbol{N}_{i,j}\}_{0\leq i<j\leq d}$ are bases of $\mathbb S$.
The basis $\{\boldsymbol{T}_{i,j}\}_{0\leq i<j\leq d}$ is introduced in \cite{Christiansen2011,Hu2015a} and $\{\boldsymbol{N}_{i,j}\}_{0\leq i<j\leq d}$ is in \cite{Christiansen2011,ChenHuHuang2018}.

\subsection{\revision{Characterization of DoFs for bubble spaces}}
\revision{We give a characterization of DoFs for bubble spaces by the dual approach and the decomposition of the bubble spaces through the bubble complex.}
\revision{
\begin{lemma}\label{lem:bubbleDoF}
Assume finite-dimensional Hilbert spaces $\mathbb B_1, \mathbb B_2, \ldots, \mathbb B_n$ with the inner product $(\cdot,\cdot)$ form an exact Hilbert complex
$$
0\xrightarrow{\subset} \mathbb B_1\xrightarrow{\dd_1} \ldots \mathbb B_i \xrightarrow{\dd_i} \ldots \mathbb B_n \rightarrow0,
$$
where $\mathbb B_i\subseteq \ker(\tr^{\dd_i})$ for $i=1,2, \cdots, n-1$.
Then the bubble space $\mathbb B_i$, for $i=1,\ldots, n-1$, is uniquely determined by the DoFs
\begin{align}
( v, \dd_i^* q) & \quad\forall~q\in \dd_i \mathbb B_i, 
\label{bubbledof1}\\
(v, q) & \quad\forall~q\in \mathbb Q \cong (\dd_{i-1}\mathbb B_{i-1})', \label{bubbledof2}
\end{align}
where $\dd_i^*$ is the adjoint of $\dd_i: \mathbb B_i\to \mathbb B_{i+1}$ with respect to the inner product $(\cdot,\cdot)$ and the isomorphism $ \mathbb Q \to (\dd_{i-1}\mathbb B_{i-1})'$ is given by $p \to (p, \cdot)$ for $p\in \mathbb Q$.
\end{lemma}
\begin{proof}
By the splitting lemma in \cite{Hatcher2002} (see also Theorem~2.2 in \cite{ChenHuang2018}),
\begin{equation}\label{bubbledecomp}
\mathbb B_i=\dd_i^*\dd_i\mathbb B_{i}\oplus \dd_{i-1}\mathbb B_{i-1}.
\end{equation}
Since $\dd_i^*$ restricted to $\dd_i\mathbb B_{i}$ is injective, the number of DoFs \eqref{bubbledof1}-\eqref{bubbledof2} is same as $\dim \mathbb B_i$.
Assume $v\in \mathbb B_i$ and all the DoFs \eqref{bubbledof1}-\eqref{bubbledof2} vanish. By the decomposition \eqref{bubbledecomp}, there exist $v_1\in\mathbb B_{i}$ and $v_2\in\mathbb B_{i-1}$ such that $v=\dd_i^*\dd_i v_1+\dd_{i-1}v_2$. The vanishing \eqref{bubbledof1} yields $\dd_i v=0$, that is $\dd_i\dd_i^*(\dd_i v_1)=0$. Noting that $\dd_i\dd_i^*: \dd_i\mathbb B_i\to \dd_i\mathbb B_i$ is isomorphic, we get $\dd_i v_1=0$ and thus $v=\dd_{i-1}v_2$. Now apply the vanishing \eqref{bubbledof2} to get $v=0$.
\end{proof}
}

\revision{When the bubble function space $\mathbb B$ can be characterized precisely, we can simply use $\mathcal N(\mathbb B)$, i.e., $(v, q), q\in \mathbb B$ as DoFs. Lemma \ref{lem:bubbleDoF} tells us it suffices to identify the dual space without knowing the explicit form of the bubble functions. In the following, we present a way to identify $\mathbb B'$ by a decomposition of the dual space.
}

\revision{
\begin{lemma} \label{lem:abstract}
Consider linear map $\dd: V\to P$ between two finite dimensional Hilbert spaces sharing the same inner product $(\cdot,\cdot)$. Let $\mathbb B = \ker(\tr^{\dd})\cap V$, $E_0 =\ker(\dd)\cap \mathbb B$, and $E_0^{\bot} = \mathbb B/E_0$.   Assume 
\begin{enumerate}[(B1)]
\smallskip
\item $\mathbb B' = \mathcal I'\mathcal N(U)$ for some subspace $U\subseteq V$; 
\smallskip
\item there exists an operator $\kappa$ leading to the inclusion: 
 \begin{equation}\label{eq:Vkappa}
U \subseteq \dd^*(H(\dd^*))\oplus ( \ker (\kappa) \cap U),
\end{equation}
where $\dd^*$ is the adjoint of $\dd: \mathbb B \to \dd \mathbb B$ with respect to the inner product $(\cdot,\cdot)$ and can be continuously extended to the space $H(\dd^*)$.
\end{enumerate}
Then 
\begin{align}
\label{eq:E0botdof} (E_0^{\bot})' &= \mathcal N(\dd^*(\dd \mathbb B)), \\
\label{eq:E0dof} E_0' &= \mathcal I'  \mathcal N( \ker (\kappa) \cap U).
\end{align}
\end{lemma}
\begin{proof}
The characterization \eqref{eq:E0botdof} is straightforward as $\dd: E_0^{\bot} \to \dd \mathbb B$ is a bijection. To prove \eqref{eq:E0dof}, it suffices to show that: for any $u\in E_0$, if $(u, p) = 0$ for all $p\in  \ker (\kappa) \cap U$, then $u = 0$. 
First of all, as $u\in E_0$, $u\perp \dd^*(H(\dd^*))$, i.e.,
$$
(u, \dd^* p) = (\dd u, p) = 0 \quad \forall~p\in H(\dd^*).
$$
Combined with the assumption \eqref{eq:Vkappa}, we have $(u,p) = 0$ for all $p\in U$ and conclude $u = 0$ by assumption $\mathbb B' = \mathcal I'\mathcal N(U)$. 
\end{proof}
}

\revision{As we mentioned before, in \eqref{eq:E0dof}, $\mathcal I'$ could be onto. For example, one can choose $U = V$. We want to choose the smallest subspace $U$ to get $E_0' = \mathcal N( \ker (\kappa) \cap U)$. One guideline is the dimension count. On one hand, we have the following identity
$$
\dim E_0 = \dim \mathbb B - \dim E_0^{\bot} =  \dim V - \dim (\img({\rm tr}^d)) - \dim E_0^{\bot}.
$$
On the other hand, we have
$$
\dim (\ker (\kappa) \cap U) = \dim U - \dim (\kappa U).
$$
For specific examples, we only need to figure out the dimension not exact identification of subspaces.
}

\revision{Next we enrich space $V$ to derive another finite element.
\begin{lemma}\label{lm:VPfem}
Consider linear map $\dd: V\to P$ between two finite dimensional Hilbert spaces sharing the same inner product $(\cdot,\cdot)$. Let $\mathbb B = \ker(\tr^{\dd})\cap V$, $E_0 =\ker(\dd)\cap \mathbb B$ and $E_0^{\bot} = \mathbb B/E_0$. Now we enrich the space to $V+\mathbb H$ and let $\mathbb B^+ = \ker(\tr^{\dd})\cap (V+ \mathbb H)$. Assume
\begin{enumerate}[(H1)]
\smallskip
\item $V\cap \mathbb H=\{0\}$ and $\dd V\cap\dd \mathbb H=\{0\}$;
\smallskip
\item $\tr^{\dd}(\mathbb H)\subseteq \tr^{\dd}(V)$;
\smallskip
\item $\dd: \mathbb H\to \dd\mathbb H$ is bijective;
\smallskip
\item $E_0 = \mathcal N(\mathbb Q), (E_0^{\bot})' = \mathcal N(\dd^*\mathbb P)$;
\item $\big(\dd \mathbb B^+ \big)'=\mathcal I'\mathcal N(\mathbb P\oplus\dd\mathbb H)$.
\end{enumerate}
Then 
\begin{equation}\label{eq:B+DoF}
(\mathbb B^+)' =  \mathcal N(\mathbb Q) + \mathcal N \big ( \dd^*(\mathbb P + \dd \mathbb H) \big ),
\end{equation}
i.e., a function $v\in \mathbb B^+$ is uniquely determined by DoFs 
\begin{align}
(v, \dd^* q) & \quad\forall~q\in \mathbb P, \label{VPbubbledof1} \\
( v, \dd^* q) & \quad\forall~q\in \dd\mathbb H, \label{VPbubbledof2}\\
(v, q) & \quad\forall~q\in \mathbb Q. \label{VPbubbledof3}
\end{align}
\end{lemma}
}
\revision{
\begin{proof}
As $\tr^{\dd}(\mathbb H)\subseteq \tr^{\dd}(V)$, $\dim \mathbb B^+ - \dim \mathbb B = \dim \mathbb H$. On the other hand, since $\dd^*\dd: \mathbb H\to \dd^*\dd\mathbb H$ is bijective, the number of DoFs increased is also $\dim \mathbb H$. Thus the dimensions in \eqref{eq:B+DoF} are equal. 
%
Take a $v\in\mathbb B^+$ and assume all the DoFs \eqref{VPbubbledof1}-\eqref{VPbubbledof3} vanish. Thanks to the vanishing DoFs \eqref{VPbubbledof1} and \eqref{VPbubbledof2}, we get from (H5) that $\dd v=0$, which together with $\dd V\cap\dd \mathbb H=\{0\}$ implies $v\in V$. Finally $v=0$ follows from the vanishing DoFs \eqref{VPbubbledof1} and \eqref{VPbubbledof3}.
\end{proof}
}

\revision{
Assumptions $(H1)-(H3)$ are built into the construction of $\mathbb H$. Assumption $(H4)$ can be verified from the characterization of $\mathbb B'$ in Lemma \ref{lem:abstract}. Only $(H5)$ requires some work. One can show the kernel of $\dd$ in the bubble space remains unchanged as $E_0$ but its imagine is enriched. The dual space is enriched from $(\dd\mathbb B)' = \mathcal N(\mathbb P)$ to $(\dd\mathbb B^+)' = \mathcal N(\mathbb P\oplus\dd\mathbb H)$. 
Note that the precise characterization of $\mathbb B^+$ is not easy and $\mathbb H$ may not be in $\mathbb B^+$. 
}
\section{Div-Conforming Finite Elements}\label{sec:divvec}
In this section we shall construct the well-known div-conforming finite elements: Brezzi-Douglas-Marini (BDM)~\cite{BrezziDouglasMarini1986,BrezziDouglasDuranFortin1987,Nedelec:1986family} and Raviart-Thomas (RT) elements~\cite{RaviartThomas1977,Nedelec1980}. We start with this simple example to illustrate our approach and build some elementary blocks. 

\subsection{Div operator}
We begin with the following result on the div operator. 
\begin{lemma}
Let integer $k\geq 0$. The mapping $\div: \boldsymbol x \mathbb H_k(D) \to \mathbb H_k(D)$ is bijective. Consequently $\div: \mathbb P_{k+1}(D;\mathbb R^d) \to \mathbb P_k(D)$ is surjective.
\end{lemma}
\begin{proof}
It is a simple consequence of the Euler's formulae \eqref{eq:homogeneouspolyprop} and \eqref{eq:Hkdiv}.
\end{proof}

\subsection{Trace space}
The trace operator for $H(\div, K)$ space 
$$
{\rm tr}^{\div}: H(\div, K) \to H^{-1/2}(\partial K)
$$
is a continuous extension of ${\rm tr}^{\div} \boldsymbol v = \boldsymbol n\cdot \boldsymbol v|_{\partial K}$ defined on smooth functions. We then focus on the restriction of the trace operator to the polynomial space. Denote by $\mathbb P_{k}(\mathcal F^1(K)):=\{q\in L^2(\partial K): q|_F\in\mathbb P_k(F) \textrm{ for each }F\in\mathcal F^1(K)\}$, which is a Hilbert space with inner product $\sum_{F\in \mathcal F^1(K)}(\cdot,\cdot)_F$. Obviously 
${\rm tr}^{\div} (\mathbb P_k(K; \mathbb R^d)) \subseteq \mathbb P_{k}(\mathcal F^1(K))$. We prove it is indeed surjective. 

\begin{lemma}\label{lem:trdivonto}
For integer $k\geq 1$, the mapping ${\rm tr}^{\div}: \mathbb P_k(K; \mathbb R^d) \to \mathbb P_{k}(\mathcal F^1(K))$ is onto. Consequently
$$
\dim {\rm tr}^{\div}(\mathbb P_k(K; \mathbb R^d)) = \dim\mathbb P_{k}(\mathcal F^1(K)) = (d+1){k + d - 1\choose k}.
$$
\end{lemma}
\begin{proof}
By the linearity of the trace operator, it suffices to prove that: for any $F_i\in \mathcal F^1(K)$ and any $p\in \mathbb P_k(F_i)$, we can find a $\boldsymbol v\in \mathbb P_k(K; \mathbb R^d)$ s.t. $\boldsymbol v\cdot \boldsymbol n|_{F_i} = p$ and $\boldsymbol v\cdot \boldsymbol n|_{F_j} = 0$ for other $F_j\in \mathcal F^1(K)$ with $j\neq i$. Without loss of generality, we can assume $i=0$. 

For any $p\in \mathbb P_k(F_0)$, it can be expanded in Bernstein basis $p = \sum\limits_{\alpha\in\mathbb N^d, |\alpha|=k}c_{\alpha}\bs \lambda^{\alpha}$, which can be naturally extended to the whole simplex by the definition of barycentric coordinates. Again by the linearity, we only need to consider one generic term still denoted by $p = c_{\alpha}\bs \lambda^{\alpha}$ for a multi-index $\alpha\in \mathbb N^d, |\alpha|=k$. 
 As $\sum\limits_{i=1}^d \alpha_i = k >0$, there exists an index $1\leq i\leq d$ s.t. $\alpha_i\neq 0$. Then we can write $p = \lambda_i q$ with $q\in \mathbb P_{k-1}(K)$. 

Now we let $\boldsymbol v = \lambda_i q \, \boldsymbol l_i /(\boldsymbol l_i, \boldsymbol n_{0})$. By construction, 
\begin{align*}
\boldsymbol v\cdot\boldsymbol n_0 &= \lambda_i q = p,\\
(\boldsymbol v\cdot\boldsymbol n_j)|_{F_{j}}&= \lambda_i q|_{F_j} \, (\boldsymbol l_i, \boldsymbol n_j) /(\boldsymbol l_i, \boldsymbol n_{0}) = 0, \quad j = 1,2,\cdots, d. 
\end{align*}
That is we find $\boldsymbol v\in\mathbb P_k(K; \mathbb R^d)$ s.t. $({\rm tr}^{\div} \boldsymbol v)|_{F_0}= p$ and $({\rm tr}^{\div} \boldsymbol v)|_{F_j}= 0$ for $j=1,\ldots, d$. 
\end{proof}

With this identification of the trace space, we clearly have $\mathcal N( \mathbb P_{k}(\mathcal F^1(K)))  = ( {\rm tr}^{\div}(\mathbb P_k(K; \mathbb R^d)))'$, and through $({\rm tr}^{\div})'$, we embed $\mathcal N( \mathbb P_{k}(\mathcal F^1(K))$ into $\mathbb P_k'(K; \mathbb R^d)$.


 
\begin{lemma}\label{lem:divtracedof}
Let integer $k\geq 1$. 
For any $\boldsymbol v\in \mathbb P_k(K; \mathbb R^d)$, if the following degrees of freedom vanish, i.e.,
\begin{equation*}
(\boldsymbol v\cdot\boldsymbol n, p)_F= 0 \quad\forall~p\in\mathbb P_{k}(F), F\in \mathcal F^1(K),
\end{equation*}
then ${\rm tr}^{\div}\boldsymbol v=0$.
\end{lemma}
\begin{proof}
Due to Lemma~\ref{lem:trdivonto}, 
 the dual operator $({\rm tr}^{\div})': \mathbb P_{k}'(\mathcal F^1(K)) \to \mathbb P_k'(K; \mathbb R^d)$ is injective. 
Taking $N=(p,\cdot)_F\in \mathbb P_{k}'(\mathcal F^1(K))$ for any $F\in \mathcal F^1(K)$ and $p\in \mathbb P_{k}(F)$, we have
$$
(({\rm tr}^{\div})'N)(\boldsymbol v)=N({\rm tr}^{\div}\boldsymbol v)=(p,{\rm tr}^{\div}\boldsymbol v)_F=(p,\boldsymbol  n\cdot\boldsymbol v)_F.
$$
By the assumption,
we have $\boldsymbol v\perp  \img(({\rm tr}^{\div})')$, which indicates $\boldsymbol v\in\ker({\rm tr}^{\div})$.
\end{proof}

\revision{Another basis of $({\rm tr}^{\div}(\mathbb P_k(K; \mathbb R^d)))'$ can be obtained by a geometric decomposition of vector Lagrange elements; see \cite{Chen;Huang:2021Geometric} for details.}

\subsection{Bubble space}
After we characterize the range of the trace operator, we focus on its null space.
Define the polynomial bubble space 
$$\mathbb B_k(\div, K) = \ker({\rm tr}^{\div})\cap \mathbb P_k(K; \mathbb R^d).$$
As $\{\boldsymbol n_i, i=1,2,\cdots, d\}$ is a basis of $\mathbb R^d$, it is obvious that for $k=0$, $\mathbb B_0(\div, K) = \{\boldsymbol0\}$. As a  direct consequence of dimension count, see Lemma \ref{lm:dimBdiv} below, $\mathbb B_1(\div, K)$ is also the zero space.

\begin{lemma}\label{lm:dimBdiv}
Let integer $k\geq 1$. It holds 
\begin{equation*}
\dim \mathbb B_k(\div, K)
= d {k + d \choose k} - (d+1){k + d - 1\choose k} = (k-1){k+d-1\choose k}.
\end{equation*}
\end{lemma}
\begin{proof}
By the characterization of the trace space, we can count the dimension
$$
\dim \mathbb B_k(\div, K)
= \dim \mathbb P_k(K;\mathbb R^d) - \dim {\rm tr}^{\div}(\mathbb P_k(K; \mathbb R^d)),
$$
\revision{as required.}
\end{proof}

Next we find different bases of $\mathbb B_k'(\div, K)$. The primary approach is to find a basis for $\mathbb B_k(\div, K)$, which induces a basis of $\mathcal N( \mathbb B_k(\div, K))$. For example, one can show
\begin{equation*}
\mathbb B_k(\div, K)=\sum_{0 \leq i<j\leq d}\lambda_i\lambda_j\mathbb P_{k-2}(K)\boldsymbol t_{i,j}\quad\textrm{ for } k\geq 2. 
\end{equation*}
Verification $\lambda_i\lambda_j\mathbb P_{k-2}(K)\boldsymbol t_{i,j}\subseteq \mathbb B_k(\div, K)$ is from the fact $$\lambda_i\lambda_j \boldsymbol t_{i,j}\cdot \boldsymbol n_{\ell}|_{F_{\ell}} = 0, \quad \ell = 0, 1, \cdots, d.$$
Indeed if $\ell = i$ or $\ell = j$, then $\lambda_i\lambda_j|_{F_{\ell}} =0$. Otherwise  $\boldsymbol t_{i,j}\cdot \boldsymbol n_{\ell} = 0$  by \eqref{eq:tijlambdal}. To show every function in $\mathbb B_k(\div, K)$ can be written as a linear combination of $\lambda_i\lambda_j \boldsymbol t_{i,j}$ is tedious and will be skipped. 
Obviously $\dim \mathbb B_k(\div, K) \neq d(d+1)/2\dim \mathbb P_{k-2}(K)$ as $\{\lambda_i\lambda_j\mathbb P_{k-2}(K)\boldsymbol t_{i,j}, 0 \leq i<j\leq d\}$ is linearly dependent. One can further expand the polynomials in $\mathbb P_{k-2}(K)$ in the Bernstein basis and $\boldsymbol t_{ij}$ in terms of $\dd \lambda$ and add constraint on the multi-index  to find a basis from this generating set; see~\cite{arnold2009geometric}. \revision{Another systematical way to identify $ \mathbb B_k(\div, K)$ is through a geometric decomposition of vector Lagrange elements and a $t-n$ basis decomposition at each sub-simplex; see \cite{Chen;Huang:2021Geometric} for details.}

Fortunately we are interested in the dual space, which can let us find a basis of $\mathbb B_k'(\div, K) $ without knowing one for $\mathbb B_k(\div, K)$. \revision{Following Lemma \ref{lem:abstract}}, we first find a larger space containing $\mathbb B_k'(\div, K)$. 
 
\begin{lemma}\label{lm:Bkdual1}
Let integer $k\geq 1$. We have
\begin{equation*}
\mathbb B_k'(\div, K) = \mathcal I'\mathcal N (\mathbb P_{k-1}(K;\mathbb R^d)),
\end{equation*}
where $\mathcal I: \mathbb B_k(\div, K)\to\mathbb P_{k}(K;\mathbb R^d)$ is the inclusion map.
\end{lemma}
\begin{proof}
It suffices to show that: for any $\boldsymbol v\in \mathbb B_k(\div, K)$ satisfying
\begin{equation}\label{eq:20210525}
(\boldsymbol v, \boldsymbol q)_K=0\quad\forall~\boldsymbol q\in\mathbb P_{k-1}(K;\mathbb R^d),
\end{equation}
then $\boldsymbol v = \boldsymbol 0$. 

Expand $\boldsymbol v$ in terms of $\{ \boldsymbol l_i \}$ as $\boldsymbol v = \sum\limits_{i=1}^d v_i \boldsymbol l_i$. Then $\boldsymbol v\cdot \boldsymbol n_i|_{F_i} = 0$ implies $v_i|_{F_i} = 0$ for $i=1,\cdots,d$, i.e., $v_i = \lambda_i p_i$ for some $p_i\in \mathbb P_{k-1}(K)$. Choose $\boldsymbol q= \sum\limits_{i=1}^d p_i \boldsymbol n_i \in \mathbb P_{k-1}(K;\mathbb R^d)$ in \eqref{eq:20210525} to get $\int_{K} \boldsymbol v \cdot \bigg(\sum\limits_{i=1}^d p_i \boldsymbol n_i\bigg) \dx = \int_{K}\sum\limits_{i=1}^d \lambda_i p_i^2 \dx = 0$, which implies $p_i = 0$, i.e. $v_i= 0$ for all $i=1,2,\cdots, d$. Thus $\boldsymbol v = \boldsymbol 0$. 
\end{proof}

Again the dimension count $\dim \mathbb B_k(\div, K) \neq \dim \mathbb P_{k-1}(K; \mathbb R^d)$ implies $\mathcal I'$ is not injective. We need to further refine our characterization.
%
%
We use $\div$ operator to decompose $\mathbb B_k(\div, K)$ into
\begin{equation*}
E_0 := \mathbb B_k(\div, K)\cap \ker(\div), \quad E_0^{\bot} := \mathbb B_k(\div, K)/ E_0.
\end{equation*}
We can characterize the dual space of $E_0^{\bot}$ through $\div^*$, which is $-\grad$ restricting to the bubble function space and \revision{can be continuously extended to $H^1$}.
\begin{lemma}
Let integer $k\geq 1$. The mapping
 $$
 \div: E_0^{\bot} \to \mathbb P_{k-1}(K)/\mathbb R
 $$
 is a bijection and consequently
 $$
 \dim E_0^{\bot} = \dim \mathbb P_{k-1}(K) - 1 = {k -1 + d \choose d} - 1.
 $$
\end{lemma}
\begin{proof}
The inclusion $ \div( \mathbb B_k(\div, K)) \subseteq \mathbb P_{k-1}(K)/\mathbb R$ is proved through integration by parts
$$
(\div \boldsymbol v, p)_K = -( \boldsymbol v, \grad p)_K = 0 \quad\forall~p\in  \mathbb R = \ker(\grad).
$$


If $ \div( \mathbb B_k(\div, K)) \neq \mathbb P_{k-1}(K)/\mathbb R$, then there exists a $p\in \mathbb P_{k-1}(K)/\mathbb R$ and $p\perp \div( \mathbb B_k(\div, K))$, which is equivalent to $\nabla p \perp \mathbb B_k(\div, K)$. Expand the vector $\nabla p$ in the basis $\{\boldsymbol n_i, i=1,\cdots, d\}$ as $\nabla p = \sum\limits_{i=1}^d q_i {\boldsymbol n}_{i}$ with $q_i\in\mathbb P_{k-2}(K)$. Then set $\boldsymbol v_p = \sum\limits_{i=1}^d q_i \lambda_{0}\lambda_i  \boldsymbol l_i  = \sum\limits_{i=1}^d|\nabla\lambda_i| q_i \lambda_{0}\lambda_i \boldsymbol t_{i,0}\in \mathbb B_k(\div, K)$. We have
$$
(\grad p, \boldsymbol v_p)_K = \sum_{i=1}^d \int_Kq_i^2 \lambda_{0}\lambda_i \dx = 0,
$$
which implies $q_i = 0$ for all $i=1,2,\cdots, d$, i.e., $\grad p = 0$ and $p = 0$ as $p\in   \mathbb P_{k-1}(K)/\mathbb R$. 

We have proved $\div( \mathbb B_k(\div, K)) = \mathbb P_{k-1}(K)/\mathbb R$, and thus $
 \div: E_0^{\bot} \to \mathbb P_{k-1}(K)/\mathbb R
 $  is a bijection as $E_0^{\bot} = \mathbb B_k(\div, K)/\ker(\div)$.
\end{proof}

\revision{As an example of \eqref{eq:E0botdof},  we have the following characterization of $(E_0^{\bot})'$.}
\begin{corollary}\label{cor:E0botstar}
Let integer $k\geq 1$. We have
\begin{equation*}
(E_0^{\bot})' = \mathcal N(\grad \mathbb P_{k-1}(K)).  
\end{equation*}
That is a function $\boldsymbol v\in E_0^{\bot}$ is uniquely determined by DoFs
\begin{equation*}
(\boldsymbol v, \boldsymbol q)_K\quad\forall~\boldsymbol q\in\grad \mathbb P_{k-1}(K).
\end{equation*}  
\end{corollary}
After we know the dimensions of $\mathbb B_k(\div, K)$ and $E_0^{\bot}$, we can calculate the dimension of $E_0$.
\begin{lemma}
Let integer $k\geq 1$. It holds
\begin{equation}\label{eq:dimE0}
\dim E_0 = \dim \mathbb B_k(\div, K) - \dim E_0^{\bot}=d {k+d-1\choose d}-{k+d\choose d} +1.
\end{equation}
\end{lemma}

The most difficult part is to characterize $E_0$. Using the de Rham complex, we can identify the null space of $\div: \mathbb B_{k}(\div, K) \to  \mathbb P_{k-1}(K)/\mathbb R$ as the image of another polynomial bubble spaces. For example,
$$
E_0=\begin{cases}
\curl\big(\mathbb P_{k+1}(K)\cap H_0^1(K)\big) & \textrm{ for } d=2,\\
\curl\big(\mathbb P_{k+1}(K;\mathbb R^3)\cap H_0(\curl,K)\big) & \textrm{ for } d=3.
\end{cases}
$$
Generalization of the $\curl$ operator can be done for de Rham complex, but it will be hard for elasticity complex and divdiv complex. 

%


Instead, we take the dual approach. 
To identify the dual space of $E_0$, we resort to a polynomial decomposition of $\mathbb P_{k-1}(K; \mathbb R^d)$. 

\begin{lemma}\label{lem:Pkdec}
Let integer $k\geq 1$. We have the polynomial space decomposition
\begin{align}
\mathbb P_{k-1}(K; \mathbb R^d) 
&= \grad \mathbb P_{k}(K)\oplus (\ker (\cdot \boldsymbol x)\cap \mathbb P_{k-1}(K;\mathbb R^d)). \label{eq:dec2}
\end{align}
\end{lemma}
\begin{proof}
Clearly it holds
$$
\grad \mathbb P_{k}(K)\oplus (\ker (\cdot \boldsymbol x)\cap \mathbb P_{k-1}(K;\mathbb R^d))\subseteq \mathbb P_{k-1}(K; \mathbb R^d).
$$
And the sum is direct as by Euler's formulae \eqref{eq:homogeneouspolyprop}, 
$$
\grad \mathbb P_{k}(K)\cap (\ker (\cdot \boldsymbol x)\cap \mathbb P_{k-1}(K;\mathbb R^d)) =\{\boldsymbol 0\}.
$$

The mapping $\cdot\bs x: \mathbb P_{k-1}(K;\mathbb R^d) \to \mathbb P_{k-1}(K)\backslash \mathbb R$ is surjective, thus
$$
\dim\mathbb P_{k-1}(K; \mathbb R^d) 
= \dim(\mathbb P_{k}(K)\backslash \mathbb R) + \dim(\ker (\cdot \boldsymbol x)\cap \mathbb P_{k-1}(K;\mathbb R^d)).
$$
As $\dim (\mathbb P_{k}(K)\backslash \mathbb R) = \dim \grad \mathbb P_{k}(K)$, we obtain the decomposition \eqref{eq:dec2} by dimension count.
\end{proof}

\begin{corollary}\label{cor:E0star}
Let integer $k\geq 1$. We have
\begin{equation*}
E_0' = \mathcal N(\ker (\cdot \boldsymbol x)\cap \mathbb P_{k-1}(K;\mathbb R^d)).
\end{equation*}
That is a function $\boldsymbol v\in E_0$ is uniquely determined by 
\begin{equation*}
(\boldsymbol v, \boldsymbol q)_K\quad\forall~\boldsymbol q\in \ker (\cdot \boldsymbol x)\cap \mathbb P_{k-1}(K;\mathbb R^d).
\end{equation*}  
\end{corollary}
\begin{proof}
%

\revision{We apply Lemma \ref{lem:abstract} with $V = \mathbb P_{k}(K;\mathbb R^d), U= \mathbb P_{k-1}(K;\mathbb R^d)$, and $\kappa = \cdot \bs x$. Lemmas \ref{lm:Bkdual1} and \ref{lem:Pkdec} verify the assumptions (B1)-(B2), and we only need to count the dimension}
\begin{align*}
\dim  (\ker (\cdot \boldsymbol x)\cap \mathbb P_{k-1}(K;\mathbb R^d))& = \dim \mathbb P_{k-1}(K;\mathbb R^d)  - \dim \mathbb P_{k}(K)+1 \\
&=d {k+d-1\choose d}-{k+d\choose d} +1 \\
&= \dim E_0,
\end{align*}
where the last step is based on \eqref{eq:dimE0}.
The desired result then follows. 
\end{proof}


\begin{remark}\rm 
 A computational approach to find an explicit basis of $\ker (\cdot \boldsymbol x)\cap \mathbb P_{k-1}(K;\mathbb R^d)$ is as follows. First find a basis for $\mathbb P_{k-1}(K;\mathbb R^d)$ and one for $\mathbb P_{k}(K)$. Then form the matrix representation $X$ of the operator $\cdot\boldsymbol x$. Afterwards the null space $\ker(X)$ can be found algebraically. $\Box$
\end{remark}

An explicit characterization of $\ker (\cdot \boldsymbol x)\cap \mathbb P_{k-1}(K;\mathbb R^d)$ is shown in~\cite[Proposition~1]{Nedelec1980}, that is $\boldsymbol v\in \ker (\cdot \boldsymbol x)\cap \mathbb H_{k-1}(K;\mathbb R^d)$ is equivalent to $\boldsymbol v\in \mathbb H_{k-1}(K;\mathbb R^d)$ such that the symmetric part of $\nabla^{k-1}\boldsymbol v$ vanishes.  
\revision{We shall give another characterization of $\ker (\cdot \boldsymbol x)\cap \mathbb P_{k-1}(K;\mathbb R^d)$.} 
\begin{lemma}\label{lem:kerx}
It holds
\begin{equation}  \label{eq:dec1}
\ker (\cdot \boldsymbol x)\cap \mathbb P_{k-1}(K;\mathbb R^d) = \mathbb P_{k-2}(K;\mathbb K)\boldsymbol x,
\end{equation}
where recall $\mathbb K$ is the subspace of skew-symmetric matrices of $\mathbb R^{d\times d}$, and
\begin{equation}\label{eq:Pkm1vecdecomp}
\mathbb P_{k-1}(K; \mathbb R^d)= \grad \mathbb P_{k}(K)\oplus\mathbb P_{k-2}(K;\mathbb K)\boldsymbol x.
\end{equation}
\end{lemma}

\begin{proof}
Clearly we have $\mathbb P_{k-2}(K;\mathbb K)\boldsymbol x\subseteq\ker (\cdot \boldsymbol x)\cap \mathbb P_{k-1}(K;\mathbb R^d)$. By \eqref{eq:dec2}, it suffices to prove \eqref{eq:Pkm1vecdecomp}.
Take $\boldsymbol q\in\mathbb P_{k-1}(K; \mathbb R^d)$. Without loss of generality, by linearity, it is sufficient to assume $\boldsymbol q=\boldsymbol x^{\alpha}\boldsymbol e_{\ell}$ with $|\alpha|=k-1$ and $1\leq\ell\leq d$. Let 
$$
p=\frac{1}{k}\boldsymbol x\cdot\boldsymbol q=\frac{1}{k}\boldsymbol x^{\alpha+\boldsymbol e_{\ell}}\in\mathbb P_{k}(K), \;\; \boldsymbol\tau=\frac{1}{k}\sum_{i=1}^d\alpha_i\boldsymbol x^{\alpha-\boldsymbol e_{i}}(\boldsymbol e_{\ell}\boldsymbol e_{i}^{\intercal}-\boldsymbol e_{i}\boldsymbol e_{\ell}^{\intercal})\in\mathbb P_{k-2}(K;\mathbb K).
$$
Then
\begin{align*}
\grad p+\boldsymbol\tau\boldsymbol x& = \frac{1}{k}\sum_{i=1}^d(\alpha_i+\delta_{i\ell})\boldsymbol x^{\alpha+\boldsymbol e_{\ell}-\boldsymbol e_{i}}\boldsymbol e_{i} + \frac{1}{k}\sum_{i=1}^d\alpha_i\boldsymbol x^{\alpha-\boldsymbol e_{i}}(\boldsymbol e_{\ell}x_{i}-\boldsymbol e_{i}x_{\ell}) \\
& = \frac{1}{k}\sum_{i=1}^d(\alpha_i+\delta_{i\ell})\boldsymbol x^{\alpha+\boldsymbol e_{\ell}-\boldsymbol e_{i}}\boldsymbol e_{i} + \frac{1}{k}\sum_{i=1}^d\alpha_i\boldsymbol x^{\alpha}\boldsymbol e_{\ell} - \frac{1}{k}\sum_{i=1}^d\alpha_i\boldsymbol x^{\alpha+\boldsymbol e_{\ell}-\boldsymbol e_{i}}\boldsymbol e_{i} \\
&=\frac{1}{k}\boldsymbol x^{\alpha}\boldsymbol e_{\ell}+\frac{|\alpha|}{k}\boldsymbol x^{\alpha}\boldsymbol e_{\ell}=\boldsymbol x^{\alpha}\boldsymbol e_{\ell}.
\end{align*}
Therefore it follows $\boldsymbol q=\grad p+\boldsymbol\tau\boldsymbol x$, i.e. 
$$
\mathbb P_{k-1}(K; \mathbb R^d)\subseteq \grad \mathbb P_{k}(K)\oplus\mathbb P_{k-2}(K;\mathbb K)\boldsymbol x.
$$ 
This combined with the fact $\grad \mathbb P_{k}(K)\oplus\mathbb P_{k-2}(K;\mathbb K)\boldsymbol x\subseteq\mathbb P_{k-1}(K; \mathbb R^d)$ gives \eqref{eq:Pkm1vecdecomp}.
\end{proof}
The decomposition \eqref{eq:dec2} and characterization \eqref{eq:dec1} can be summarized as the following double-directional complex 
\begin{equation*}
\xymatrix{
\mathbb R\ar@<0.4ex>[r]^-{\subset} & \mathbb P_{k}(K)\ar@<0.4ex>[r]^-{\nabla} \ar@<0.4ex>[l]^-{\pi_0}  & \mathbb P_{k-1}(K;\mathbb R^d)\ar@<0.4ex>[r]^-{\skw\nabla}\ar@<0.4ex>[l]^-{\boldsymbol v\cdot\boldsymbol x}  & \mathbb P_{k-2}(K;\mathbb K) \ar@<0.4ex>[l]^-{\boldsymbol\tau\boldsymbol x}.
}
\end{equation*}

Define, for an integer $k\geq 0$,
\begin{equation*}
{\rm ND}_{k}(K):=\mathbb P_{k}(K;\mathbb R^d) \oplus \mathbb H_{k}(K; \mathbb K)\boldsymbol x,
\end{equation*}
which is the shape function space of the first kind Ned\'el\'ec edge element in arbitrary dimension~\cite{Nedelec1980,ArnoldFalkWinther2006}.

\begin{corollary}\label{cor:bdmbubbledual}
Let integer $k\geq 2$. We have
\begin{equation}\label{eq:Bstar}
\mathbb B_k'(\div, K)  = \mathcal N (\grad \mathbb P_{k-1}(K) \oplus \mathbb P_{k-2}(K;\mathbb K)\boldsymbol x)= \mathcal N({\rm ND}_{k-2}(K)).
\end{equation}
\end{corollary}
\begin{proof}
The first identity is a direct consequence of Corollaries \ref{cor:E0botstar} and \ref{cor:E0star} and Lemma \ref{lem:kerx}. We then write $\mathbb P_{k-2}(K;\mathbb K) = \mathbb P_{k-3}(K;\mathbb K)\oplus \mathbb H_{k-2}(K;\mathbb K)$ and use the decomposition \eqref{eq:dec2} to conclude the second identity. 
\end{proof}
%
%

\begin{remark} \rm 
The space 
${\rm ND}_{k-2}(K)$ can be abbreviate as $\mathbb P_{k-1}^{-}\Lambda^1$ in the terminology of FEEC~\cite{ArnoldFalkWinther2006,Arnold2018}. The characterization of $\mathbb B_k'(\div, K)$ in \eqref{eq:Bstar} can be written as 
$$
(\stackrel{\circ}{\mathbb P_k}\Lambda^{n-1}(K))^* = \mathbb P_{k-1}^{-}\Lambda^1(K),
$$
which is well documented for de Rham complex~\cite{arnold2009geometric} but not easy for general complexes. Therefore we still stick to the vector/matrix calculus notation. $\Box$
\end{remark}

Hence we acquire the uni-solvence of BDM element from Lemma~\ref{lem:divtracedof} and Corollary~\ref{cor:bdmbubbledual}.
\begin{theorem}[BDM element]\label{thm:BDMunisolvence}
Let integer $k\geq 1$. Choose the shape function space $V = \mathbb P_k(K;\mathbb R^d)$. We have the following set of degrees of freedom for $V$ 
\begin{align}
(\boldsymbol  v\cdot\boldsymbol n, q)_F& \quad\forall~q\in\mathbb P_{k}(F), F\in \mathcal F^1(K), \label{eq:BDMdof1}\\
(\boldsymbol v, \boldsymbol q)_K &\quad\forall~\boldsymbol q\in{\rm ND}_{k-2}(K)= \grad \mathbb P_{k-1}(K)\oplus\mathbb P_{k-2}(K;\mathbb K)\boldsymbol x.\label{eq:BDMdof2}
\end{align}

\end{theorem}
Define the global BDM element space
\begin{align*}
\boldsymbol V_h:=\{\boldsymbol v\in \boldsymbol L^2(\Omega;\mathbb R^d): &\,\boldsymbol v|_K\in\mathbb P_k(K;\mathbb R^d) \textrm{ for each } K\in\mathcal T_h, \\
&\textrm{ the degree of freedom \eqref{eq:BDMdof1} is single-valued} \}.    
\end{align*}
Since $\boldsymbol  v\cdot\boldsymbol n|_F\in\mathbb P_k(F)$, the single-valued degree of freedom \eqref{eq:BDMdof1} implies $\boldsymbol  v\cdot\boldsymbol n$ is continuous across the boundary of elements, hence $\boldsymbol V_h\subset\boldsymbol H(\div,\Omega)$. 



\subsection{RT space}
In this subsection, we assume integer $k\geq 0$.
The space of shape functions for the Raviart-Thomas (RT) element is enriched to \begin{equation*}
V^{\rm RT}:= \mathbb P_{k}(K;\mathbb R^d)\oplus\mathbb H_{k}(K)\boldsymbol x.
\end{equation*}
The degrees of freedom are
\begin{align}
(\boldsymbol v\cdot\boldsymbol n, q)_F& \quad\forall~q\in\mathbb P_{k}(F), F\in \mathcal F^1(K), \label{eq:RTdof1}\\
(\boldsymbol v, \boldsymbol q)_K &\quad\forall~\boldsymbol q\in\mathbb P_{k-1}(K,\mathbb R^d)= \grad \mathbb P_{k}(K)\oplus\mathbb P_{k-2}(K;\mathbb K)\boldsymbol x.\label{eq:RTdof2}
\end{align}

Note that the DoFs to determine the trace remain the same and only the interior moments are increased from ${\rm ND}_{k-2}(K)$ to $\mathbb P_{k-1}(K;\mathbb R^d)$. The range space is also increased, i.e., $\div V^{\rm RT} = \mathbb P_{k}(K)$ and therefore the approximation of $\div \boldsymbol u$ will be one order higher. 

We follow our construction procedure to identify the dual spaces of each block. 
\begin{lemma}
Let integer $k\geq 0$. It holds
 \begin{equation}\label{eq:RTtrace}
 {\rm tr}^{\div}(V^{\rm RT}) = \mathbb P_{k}(\mathcal F^1(K)). 
\end{equation}
\end{lemma}
\begin{proof}
When $k\geq1$, by Lemma~\ref{lem:trdivonto} and the fact that $\boldsymbol n \cdot \boldsymbol x|_{F_i}$ is a constant, it follows 
$$
 {\rm tr}^{\div}(V^{\rm RT}) = {\rm tr}^{\div}(\mathbb P_k(K; \mathbb R^d)) = \mathbb P_{k}(\mathcal F^1(K)). 
$$

Consider the case $k=0$. It is clear that ${\rm tr}^{\div}(V^{\rm RT}) \subseteq\mathbb P_{0}(\mathcal F^1(K))$.
To prove the other side, by the linearity, assume $q\in\mathbb P_{0}(\mathcal F^1(K))$ such that $q|_{F_0}=c\in\mathbb R$ and $q|_{F_i}=0$ for $i=1,\cdots, d$. Set $\boldsymbol v=\frac{c}{(\boldsymbol x_1-\boldsymbol x_0)\cdot\boldsymbol n_0}(\boldsymbol x-\boldsymbol x_0)\in V^{\rm RT}$, then ${\rm tr}^{\div}\boldsymbol v=q$. 
\end{proof}

Define the bubble space
$$
\mathbb B_{k+1}^{-}(\div, K) := \ker ( {\rm tr}^{\div} ) \cap V^{\rm RT}.
$$
By \eqref{eq:RTtrace}, $\dim\mathbb B_{k+1}^{-}(\div, K)=\dim V^{\rm RT}-\dim {\rm tr}^{\div}(V^{\rm RT})=d {k+d-1\choose d}$ for $k\geq1$ and $\dim\mathbb B_{1}^{-}(\div, K)=0$.
We show the intersection of the null space of $\div$ operator and $\mathbb B_{k+1}^{-}(\div, K)$ remains unchanged.
\begin{lemma}
Let integer $k\geq 0$. It holds
\begin{equation}\label{eq:E0RT}
\mathbb B_{k+1}^{-}(\div, K) \cap \ker (\div) = E_0,
\end{equation}
where $E_0 := \{\boldsymbol0\}$ for $k=0$.
\end{lemma}
\begin{proof}
For $\boldsymbol v\in V^{\rm RT}$, if $\div\boldsymbol v = 0$, then $\boldsymbol v\in \mathbb P_{k}(K;\mathbb R^d)$ as $\div: \mathbb H_k(K) \boldsymbol x \to \mathbb H_k(K)$ is bijective. Then the desired result follows.  
\end{proof}

Define $E_{0}^{\bot,-} := \mathbb B_{k+1}^{-}(\div, K) /E_0$. We give a characterization of $(E_{0}^{\bot,-})'$.

\begin{lemma}
Let integer $k\geq 0$. It holds
\begin{equation}\label{eq:E0perpRT}
(E_{0}^{\bot,-})' = \mathcal N(\grad \mathbb P_k(K)).
\end{equation}
\end{lemma}
\begin{proof}
We first prove: given a $\boldsymbol v\in E_{0}^{\bot,-}$, i.e. $\tr^{\div} \boldsymbol v = 0$ and $\boldsymbol v \perp E_0$, if 
\begin{equation}\label{eq:vgradp}
(\boldsymbol v, \grad p) = 0 \quad \forall~p\in \mathbb P_k(K),
\end{equation}
then $\boldsymbol v = \boldsymbol0$. 
Indeed integration by parts of \eqref{eq:vgradp} and the fact $\div \boldsymbol v\in \mathbb P_k(K)$ imply $\div\boldsymbol v = 0$, i.e. $\boldsymbol v\in E_0$. Then the only possibility to have $\boldsymbol v \perp E_0$ is $\boldsymbol v = \boldsymbol0$.

Then the dimension count gives
$$
\dim E_{0}^{\bot,-}=\dim\mathbb B_{k+1}^{-}(\div, K)-\dim E_0={k+d\choose d} -1=\dim\grad \mathbb P_k(K),
$$
which indicates \eqref{eq:E0perpRT}.
\end{proof}

Hence we acquire the uni-solvence of RT element from 
\eqref{eq:RTtrace}, \eqref{eq:E0RT}, \eqref{eq:E0perpRT} and Corollary~\ref{cor:E0star}. Global version of finite element space can be defined similarly. 
\begin{theorem}[Uni-solvence of RT element]\label{thm:RTunisolvence}
Let integer $k\geq 0$. The degrees of freedom \eqref{eq:RTdof1}-\eqref{eq:RTdof2} are uni-solvent for $V^{\rm RT}$.
\end{theorem}

\revision{When $k\geq 1$, RT element can be enriched from BDM element by applying Lemma~\ref{lm:VPfem} with $\dd=\div$, $V=\mathbb P_{k}(K;\mathbb R^d)$, $\mathbb H=\mathbb H_{k}(K)\boldsymbol x$, $\mathbb P=\mathbb P_{k-1}(K)/\mathbb R$ and $\mathbb Q=\mathbb P_{k-2}(K;\mathbb K)\boldsymbol x$.}

\section{Symmetric Div-Conforming Finite Elements}\label{sec:divS}
In this section we shall construct div-conforming finite elements for symmetric matrices. For space $V = \mathbb P_k(K,\mathbb S)$, our element is slightly different from Hu's element constructed in~\cite{Hu2015a}. A new family of $\mathbb P_{k+1}^{-}(K,\mathbb S)$ type finite elements is also constructed. The trace space for symmetric div-conforming element seems hard to characterize, instead we identify the bubble function space and then only need to work on the dual of the trace space.

\subsection{Div operator}
\begin{lemma}\label{lem:divsymtensoronto}
Let $k\geq 0$. The operator $\div: \sym (\mathbb H_{k}(D;\mathbb R^d)\boldsymbol x^{\intercal}) \to \mathbb H_{k}(D;\mathbb R^d)$ is bijective and consequently $\div: \mathbb P_{k+1}(D;\mathbb S) \to \mathbb P_{k}(D;\mathbb R^d)$ is surjective.
\end{lemma}
\begin{proof}
Noting that
\begin{align*}
\div(\sym(\mathbb H_{k}(D;\mathbb R^d)\boldsymbol x^{\intercal})) \subseteq \mathbb H_{k}(K;\mathbb R^d), \\
 \dim(\sym(\mathbb H_{k}(D;\mathbb R^d)\boldsymbol x^{\intercal}))=\dim\mathbb H_{k}(K;\mathbb R^d),
\end{align*}
it is sufficient to prove $\sym(\mathbb H_{k}(D;\mathbb R^d)\boldsymbol x^{\intercal})\cap\ker(\div)=\{\boldsymbol0\}$. That is: for any $\boldsymbol q\in\mathbb H_{k}(D;\mathbb R^d)$ satisfying $\div\sym(\boldsymbol q\boldsymbol x^{\intercal})=\boldsymbol0$, we are going to prove $\boldsymbol q =\boldsymbol 0$. 

By \eqref{eq:Hkdiv}, we have
\begin{align*}
2\div \sym(\boldsymbol q\boldsymbol x^{\intercal}) & = \div (\boldsymbol q\boldsymbol x^{\intercal}) +\div (\boldsymbol x\boldsymbol q^{\intercal}) = (k+d)\boldsymbol q + (\grad \boldsymbol x) \boldsymbol q + (\div \boldsymbol q) \boldsymbol x \\
&= (k+d+1)\boldsymbol q+ (\div \boldsymbol q) \boldsymbol x.
\end{align*}
It follows from $\div \sym(\boldsymbol q\boldsymbol x^{\intercal})=\boldsymbol 0$ that 
\begin{equation}\label{eq:divsymxq}
(k+d+1)\boldsymbol q+ (\div \boldsymbol q) \boldsymbol x=\boldsymbol 0.
\end{equation}
Applying the divergence operator $\div$ on both side of \eqref{eq:divsymxq}, we get from \eqref{eq:Hkdiv} that
$$
2(k+d)\div\boldsymbol q=0.
$$
Hence $\div\boldsymbol q=0$, which together with \eqref{eq:divsymxq} gives $\boldsymbol q=\boldsymbol 0$.
\end{proof}

\subsection{Bubble space}
Define an $\boldsymbol{H}(\div, K; \mathbb{S})$ bubble function space of polynomials of degree $k$ as
$$
\mathbb{B}_{k}(\div, K; \mathbb{S}):=\left\{\boldsymbol{\tau}\in \mathbb{P}_{k}(K; \mathbb{S}): \boldsymbol{\tau}\boldsymbol{n}|_{\partial K}=\boldsymbol{0}\right\}.
$$
It is easy to check that $\mathbb{B}_{1}(\div, K; \mathbb{S})$ is merely the zero space. The following characterization of $\mathbb{B}_{k}(\div, K; \mathbb{S})$ is given in~\cite[Lemma 2.2]{Hu2015a}. 
\begin{lemma}\label{lm:bubbledof}
For $k\geq 2$, it holds
\begin{equation}\label{eq:symtensorbubble}
    \mathbb{B}_{k}(\div,K; \mathbb{S})=\sum_{0\leq i<j\leq d}\lambda_i\lambda_j \mathbb P_{k-2}(K)\boldsymbol{T}_{i,j}.
\end{equation}
Consequently $$\dim \mathbb{B}_{k}(\div,K; \mathbb{S}) = \dim \mathbb P_{k-2}(K;\mathbb S) = \frac{d(d+1)}{2} {d+k-2 \choose d}.$$
\end{lemma}


\begin{lemma}\label{lem:HuZhanginteriordof}
For $k\geq2$,  it holds
$$
\mathbb{B}_{k}'(\div, K; \mathbb{S}) = \mathcal N (\mathbb{P}_{k-2}(K; \mathbb{S})).
$$
That is $\boldsymbol{\tau}\in\mathbb{B}_{k}(\div, K; \mathbb{S})$ is uniquely determined by 
$$
(\boldsymbol{\tau}, \boldsymbol{\varsigma})_K \quad\forall~\boldsymbol{\varsigma}\in\mathbb{P}_{k-2}(K; \mathbb{S}).
$$ 
\end{lemma}
\begin{proof}
Given $\boldsymbol \tau \in \mathbb{B}_{k}(\div, K; \mathbb{S})$, by \eqref{eq:symtensorbubble}, there exist $q_{ij}\in \mathbb P_{k-2}(K)$ with $0\leq i<j\leq d$ such that 
$$
\boldsymbol{\tau}=\sum_{0\leq i<j\leq d}\lambda_i\lambda_jq_{ij}\boldsymbol{T}_{i,j}.
$$
Note that symmetric tensors $\{\boldsymbol{N}_{i,j}\}_{0\leq i<j\leq d}$ are dual to $\{\boldsymbol{T}_{i,j}\}_{0\leq i<j\leq d}$ with respect to the Frobenius inner product (cf.~\cite[Section~3.1]{ChenHuHuang2018} and also \S \ref{sec:tensor}).
Choosing $\boldsymbol{\varsigma}=\sum\limits_{0\leq i<j\leq d} q_{ij}\boldsymbol{N}_{i,j}\in\mathbb{P}_{k-2}(K; \mathbb{S})$, we get
$$
(\boldsymbol \tau, \boldsymbol{\varsigma})_K = \sum_{0\leq i<j\leq d}(\lambda_i\lambda_j, q^2_{ij})_K=0.
$$
Hence $q_{ij}=0$ for all $i,j$, and then $\boldsymbol{\tau}=\boldsymbol{0}$. As the dimensions match, we conclude the result.
\end{proof}

\revision{Another characterization of $\mathbb B_k(\div, K;\mathbb S)$ and $\mathbb B_k'(\div, K;\mathbb S)$ is given in \cite{Chen;Huang:2021Geometric}.}

\subsection{Trace spaces}
The mapping ${\rm tr}^{\div}: \mathbb P_k(K; \mathbb S) \to \mathbb P_{k}(\mathcal F^1(K;\mathbb R^{d-1}))$ 
is not onto due to the symmetry. Some compatible conditions should be imposed on lower dimensional simplexes. Fortunately, we only need its dimension.

\begin{lemma} Let integer $k\geq 1$. It holds 
\begin{align*}
\dim {\rm tr}^{\div}( \mathbb P_k(K; \mathbb S)) &= \dim \mathbb P_k(K; \mathbb S) - \dim \mathbb{B}_{k}(\div,K; \mathbb{S}) \\
&= \dim \mathbb H_k(K; \mathbb S) + \dim \mathbb H_{k-1}(K; \mathbb S)\\
&= \frac{1}{2}d(d+1) \left [{d+k-1 \choose d-1} + {d+k-2 \choose d-1}\right ].
\end{align*}
\end{lemma}

We show the super-smoothness induced by the symmetry for $H(\div;\mathbb S)$ element. 
For a $(d-r)$-dimensional face $e\in \mathcal F^{r}(K)$ with $r=2,\cdots,d$ shared by two $(d-1)$-dimensional faces $F, F'\in \mathcal F^{1}(K)$, by the symmetry of $\boldsymbol \tau$, $(\boldsymbol n_{F}^{\intercal}\boldsymbol \tau\boldsymbol n_{F'})|_e$ is concurrently determined by $(\boldsymbol \tau\boldsymbol n_{F})|_{F}$ and $(\boldsymbol \tau\boldsymbol n_{F'})|_{F'}$. This implies the degrees of freedom $\boldsymbol n_{i}^{\intercal}\boldsymbol \tau\boldsymbol n_{j}$ on $e$ for all $i,j=1,\cdots, r$.
In particular, for a $0$-dimensional vertex $\delta$, $(\boldsymbol \tau_{ij}(\delta))_{d\times d}$ is taken as a degree of freedom. 

The trace $\boldsymbol \tau \boldsymbol n$ restricted to a face $F\in \mathcal F^1(K)$ can be further split into two components: 1) the normal-normal component $\boldsymbol n^{\intercal}\boldsymbol \tau\boldsymbol n$ will be determined by $\boldsymbol n_{i}^{\intercal}\boldsymbol \tau\boldsymbol n_{j}$; 2) the tangential-normal component $\Pi_F\boldsymbol \tau\boldsymbol n$ will be determined by the interior moments relative to $F$ after the trace $\tr^{\div_F}(\Pi_F\boldsymbol \tau\boldsymbol n) = \boldsymbol n_{F,e}^{\intercal}\boldsymbol \tau\boldsymbol n$ has been determined.


\begin{lemma}\label{lem:divSboundarydofs}
A basis of
 $$
 ( {\rm tr}^{\div}( \mathbb P_k(K; \mathbb S)) )'
 $$
 is given by the degrees of freedom 
\begin{align}
\boldsymbol \tau (\delta) & \quad\forall~\delta\in \mathcal V(K), \label{HdivSfemdof1}\\
(\boldsymbol  n_i^{\intercal}\boldsymbol \tau\boldsymbol n_j, q)_F & \quad\forall~q\in\mathbb P_{k+r-d-1}(F),  F\in\mathcal F^r(K),\;  \label{HdivSfemdof2}\\
&\quad\quad i,j=1,\cdots, r, \textrm{ and } r=1,\cdots, d-1, \notag\\
(\Pi_F\boldsymbol \tau\boldsymbol n, \boldsymbol q)_F & \quad\forall~\boldsymbol q\in {\rm ND}_{k-2}(F),  F\in\mathcal F^1(K).\label{HdivSfemdof3}
\end{align}
\end{lemma}
\begin{proof}
We first prove that if all the degrees of freedom \eqref{HdivSfemdof1}-\eqref{HdivSfemdof3} vanish, then $\boldsymbol \tau = \boldsymbol0$.  
As $\boldsymbol  n_i^{\intercal}\boldsymbol \tau\boldsymbol n_j|_F\in\mathbb P_k(F) $, 
by the vanishing degrees of freedom \eqref{HdivSfemdof1}-\eqref{HdivSfemdof2} and the uni-solvence of the Lagrange element, we get
$$
\boldsymbol n_i^{\intercal}\boldsymbol\tau\boldsymbol n_j|_F=0 \quad\forall~F\in\mathcal F^r(K), \; i,j=1,\cdots, d-r, \textrm{ and } r=1,\cdots, d-1.
$$
This implies
\begin{equation}\label{eq:20210603}
\boldsymbol n^{\intercal}\boldsymbol\tau\boldsymbol n|_F=0, \; \boldsymbol n_{F,e}^{\intercal}\boldsymbol\tau\boldsymbol n|_e=0 \quad\forall~F\in\mathcal F^1(K), e\in\mathcal F^1(F).  
\end{equation}
Notice that $\Pi_F\boldsymbol\tau\boldsymbol n|_F\in \mathbb P_k(F; \mathbb{R}^{d-1})$. Due to the uni-solvence of the BDM element on $F$, cf. Theorem~\ref{thm:BDMunisolvence}, we acquire from the second \revision{identity} in \eqref{eq:20210603} and the vanishing degrees of freedom \eqref{HdivSfemdof3} that $\Pi_F\boldsymbol\tau\boldsymbol n|_F=\boldsymbol 0$, which together with the first \revision{identity} in~\eqref{eq:20210603} yields $\boldsymbol\tau\boldsymbol n|_F=\boldsymbol0$.

We then count the dimension to finish the proof. \revision{By comparing DoFs of Hu element (cf. Remark~\ref{hudofs}) and DoFs \eqref{HdivSfemdof1}-\eqref{HdivSfemdof3}, it follows from the DoFs of the first kind N\'ed\'elec element, cf. \cite{Nedelec1980,ArnoldFalkWinther2006}, that the number of DoFs \eqref{HdivSfemdof1}-\eqref{HdivSfemdof3} is equal to the number of DoFs of Hu element, thus equals to $\dim {\rm tr}^{\div}( \mathbb P_k(K; \mathbb S))$}.
\end{proof}

\begin{remark}\rm
\label{hudofs}
As a comparison, the degrees of freedom of Hu element on boundary in~\cite{Hu2015a}
are
\begin{align*}
\boldsymbol \tau (\delta) & \quad\forall~\delta\in \mathcal V(K), \\
(\boldsymbol  n_i^{\intercal}\boldsymbol \tau\boldsymbol n_j, q)_F & \quad\forall~q\in\mathbb P_{k+r-d-1}(F),  F\in\mathcal F^r(K),\;  \\
&\quad\quad i,j=1,\cdots, r, \textrm{ and } r=1,\cdots, d-1, \notag\\
(\boldsymbol t_i^{\intercal}\boldsymbol \tau\boldsymbol n_j, q)_F & \quad\forall~q\in\mathbb P_{k+r-d-1}(F),  F\in\mathcal F^r(K),\;  \\
&\quad\quad i=1,\cdots, d-r,j=1,\cdots, r, \textrm{ and } r=1,\cdots, d-1.
\end{align*}
The difference is the way to impose the tangential-normal component. $\Box$
\end{remark}

\subsection{Split of the bubble space}\label{subsec:newdivSfem}
To construct $\boldsymbol H(\div, K;\mathbb S)$ elements, the interior degrees of freedom  given by $\mathcal N (\mathbb{P}_{k-2}(K; \mathbb{S}))$ are enough. For the construction of $\boldsymbol H(\div\div, K;\mathbb S)$ element, we use $\div$ operator to decompose $\mathbb B_k(\div, K;\mathbb S)$ into
\begin{equation*}
E_{0,k}(\mathbb S) := \mathbb B_k(\div, K; \mathbb S)\cap \ker(\div), \quad E_{0,k}^{\bot}(\mathbb S) := \mathbb B_k(\div, K; \mathbb S)/ E_{0,k}(\mathbb S).
\end{equation*}
We will abbreviate $E_{0,k}(\mathbb S)$ and $E_{0,k}^{\bot}(\mathbb S)$ as $E_{0}(\mathbb S)$ and $E_{0}^{\bot}(\mathbb S)$ respectively if not causing any confusion.
As before we can characterize the dual space of $E_{0,k}^{\bot}(\mathbb S)$ through $\div^*$, which is $-\defm:=-\sym \grad$ restricting to the bubble space and can be extended to $H^1(K; \mathbb R^d)$.
\begin{lemma}\label{lem:E0Sbot}
Let integer $k\geq2$. The mapping
 $$
 \div: E_{0,k}^{\bot}(\mathbb S) \to \mathbb P_{k-1,{\rm RM}}^{\bot}:=\mathbb P_{k-1}(K,\mathbb R^d)/\ker(\defm)
 $$
 is a bijection and consequently
\begin{align*}
(E_{0,k}^{\bot}(\mathbb S))' &= \mathcal N(\defm \mathbb P_{k-1}(K,\mathbb R^d)),\\
\dim E_{0,k}^{\bot}(\mathbb S)  &=  d{k+d-1\choose k-1} -\frac{1}{2}(d^2+d).
\end{align*}
\end{lemma}
\begin{proof}
The fact $\div\mathbb B_k(\div, K; \mathbb S)=\mathbb P_{k-1,{\rm RM}}^{\bot}$ was proved in~\cite[Theorem~2.2]{Hu2015a}. Here we recall it for the completeness. 

The inclusion $ \div( \mathbb B_k(\div, K;\mathbb S)) \subseteq \mathbb P_{k-1,{\rm RM}}^{\bot}$ can be proved through integration by parts
$$
(\div \boldsymbol \tau, \boldsymbol v)_K = -( \boldsymbol \tau, \defm \boldsymbol v)_K = 0 \quad \forall~\boldsymbol v\in \ker(\defm).
$$
If $ \div( \mathbb B_k(\div, K;\mathbb S)) \neq \mathbb P_{k-1,{\rm RM}}^{\bot}$, then there exists a function $\boldsymbol v\in  \mathbb P_{k-1,{\rm RM}}^{\bot}$ satisfying $\boldsymbol v\perp \div( \mathbb B_k(\div, K; \mathbb S))$, which is equivalent to $\defm \boldsymbol v \perp \mathbb B_k(\div, K; \mathbb S)$. Expand the symmetric matrix $\defm \boldsymbol v $ in the basis $\{\boldsymbol N_{ij}, 0\leq i<j\leq d\}$ as $\defm \boldsymbol v = \sum\limits_{0\leq i<j\leq d} q_{ij} {\boldsymbol N}_{ij}$ with $q_{ij}\in\mathbb P_{k-2}(K)$. Then set $\boldsymbol \tau_v = \sum\limits_{0\leq i<j\leq d} q_{ij} \lambda_i \lambda_j \boldsymbol T_{ij} \in \mathbb B_k(\div, K;\mathbb S)$. We have
$$
(\defm \boldsymbol v, \boldsymbol \tau_v)_K = \sum_{0\leq i<j\leq d} \int_Kq_{ij}^2 \lambda_{i}\lambda_j \dx = 0,
$$
which implies $q_{ij} = 0$ for all $0\leq i< j \leq d$, i.e., $\defm \boldsymbol v = 0$ and $\boldsymbol v = 0$ as $\boldsymbol v\in  \mathbb P_{k-1,{\rm RM}}^{\bot}$. 
%
%
Since $\div E_{0,k}^{\bot}(\mathbb S)=\div\mathbb B_k(\div, K; \mathbb S)$, the mapping
 $\div: E_{0,k}^{\bot}(\mathbb S) \to \mathbb P_{k-1,{\rm RM}}^{\bot}$
 is a bijection.

 For $\boldsymbol v\in E_{0,k}^{\bot}(\mathbb S)$, $(\boldsymbol v, \defm\boldsymbol q)_K = 0$ for all $\boldsymbol q\in \mathbb P_{k-1}(K,\mathbb R^d)$ implies $\div \boldsymbol v = \boldsymbol 0$, i.e., $\boldsymbol v\in E_{0,k}(\mathbb S)$. Then $\boldsymbol v\in E_{0,k}(\mathbb S)\cap E_{0,k}^{\bot}(\mathbb S) = \{\boldsymbol 0\}$. Hence $( E_{0,k}^{\bot}(\mathbb S))' = \mathcal I' \mathcal N(\defm \mathbb P_{k-1}(K,\mathbb R^d))$.  As the dimensions match, $\mathcal I'$ is a bijection.
\end{proof}


We then move to the space $E_{0,k}(\mathbb S) $. Using the primary approach, we need the bubble space in the previous space and the differential operator. For example, we have
$
E_{0,k}(\mathbb S)=\curl\curl(\mathbb P_{k+2}(K)\cap H_0^2(K))
$
in two dimensions~\cite{ArnoldWinther2002}, and in three dimensions~\cite{ArnoldAwanouWinther2008,Chen;Huang:2021Finite}
$$
E_{0,k}(\mathbb S)=\textrm{inc}\, \mathbb B_{k+2} ({\rm inc},K;\mathbb S)
$$
with
\begin{align*}
\mathbb B_{k+2} ({\rm inc},K;\mathbb S):=\{&\boldsymbol\tau\in\mathbb P_{k+2}(K;\mathbb S): \boldsymbol n\times\boldsymbol\tau\times\boldsymbol n=\boldsymbol 0, \\
&2\defm_F(\boldsymbol n\cdot\boldsymbol\tau\Pi_F)-\Pi_F\partial_n\boldsymbol\tau\Pi_F=\boldsymbol 0\;\;\forall~F\in\mathcal F^1(K) \}.
\end{align*} 
Such characterization is hard to be generalized to arbitrary dimension.

Instead we use the dual approach to identify $E_{0,k}'(\mathbb S)$. 
To this end, denote the space of rigid motions by
\[
\boldsymbol{RM}:={\rm ND}_{0}(K)=\{\boldsymbol c+\boldsymbol N\boldsymbol x: \;\boldsymbol c\in\mathbb R^d,\; \boldsymbol N\in \mathbb K\}.
\]
Define operator $\boldsymbol \pi_{RM}: \mathcal C^1(D; \mathbb R^d)\to \boldsymbol{RM}$ as
\[
\boldsymbol \pi_{RM}\boldsymbol  v:=\boldsymbol  v(\boldsymbol 0)+(\skw(\nabla\boldsymbol v))(\boldsymbol 0)\boldsymbol x.
\]
Clearly it holds
\begin{equation*}
\boldsymbol \pi_{RM}\boldsymbol  v=\boldsymbol  v\quad \forall~\boldsymbol  v\in\boldsymbol{RM}.
\end{equation*}

We shall establish the following short exact sequence
\begin{equation*}
\xymatrix{
\boldsymbol{RM}\ar@<0.4ex>[r]^-{\subset} & \mathbb P_{k+1}(D;\mathbb R^d)\ar@<0.4ex>[r]^-{\defm} \ar@<0.4ex>[l]^-{\boldsymbol \pi_{RM}}  & \defm\mathbb P_{k+1}(D; \mathbb S)\ar@<0.4ex>[l]^-{\boldsymbol \tau\boldsymbol x}  
\ar@<0.4ex>[r]^-{}& \boldsymbol0 \ar@<0.4ex>[l]^-{}
}
\end{equation*}
and derive a space decomposition from it. 

\begin{lemma}\label{lem:xdefkernal}
Let integer $k\geq0$. If $\boldsymbol q\in \mathbb P_{k+1}(D;\mathbb R^d)$ satisfying $(\defm\boldsymbol q)\boldsymbol x=\boldsymbol0$, then $\boldsymbol q\in\boldsymbol{RM}$.
\end{lemma}
\begin{proof}
Since
$
\boldsymbol x^{\intercal}(\boldsymbol x\cdot\nabla)\boldsymbol q=\boldsymbol x^{\intercal}(\nabla\boldsymbol q)\boldsymbol x=\boldsymbol x^{\intercal}(\defm\boldsymbol q)\boldsymbol x=0
$,
we get
$$(\boldsymbol x\cdot\nabla)(\boldsymbol x^{\intercal}\boldsymbol q)=\boldsymbol x^{\intercal}(\boldsymbol x\cdot\nabla)\boldsymbol q+\boldsymbol x^{\intercal}\boldsymbol q=\boldsymbol x^{\intercal}\boldsymbol q.$$
By \eqref{eq:homogeneouspolyprop}, this indicates $\boldsymbol x^{\intercal}\boldsymbol q\in\mathbb P_1(D)$.
Noting that $(\nabla\boldsymbol q)\boldsymbol x = \nabla(\boldsymbol x^{\intercal}\boldsymbol q)-\boldsymbol q$, we obtain
$$(\boldsymbol x\cdot\nabla)\boldsymbol q + (\nabla(\boldsymbol x^{\intercal}\boldsymbol q)-\boldsymbol q)=(\nabla\boldsymbol q)^{\intercal}\boldsymbol x + (\nabla\boldsymbol q)\boldsymbol x=2(\defm\boldsymbol q)\boldsymbol x=\boldsymbol0,$$
which implies $(\boldsymbol x\cdot\nabla)\boldsymbol q - \boldsymbol q = -\nabla(\boldsymbol x^{\intercal}\boldsymbol q)\in \mathbb P_0(D;\mathbb R^d)$.
Hence $\boldsymbol q\in\mathbb P_1(D;\mathbb R^d)$. Assume $\boldsymbol q=\boldsymbol N\boldsymbol x+\boldsymbol C$ with $\boldsymbol N\in\mathbb M$ and $\boldsymbol C\in\mathbb R^d$. Then
$$
\boldsymbol x^{\intercal}(\sym\boldsymbol N)\boldsymbol x+\boldsymbol x^{\intercal}\boldsymbol C=\boldsymbol x^{\intercal}\boldsymbol N\boldsymbol x+\boldsymbol x^{\intercal}\boldsymbol C=\boldsymbol x^{\intercal}\boldsymbol q\in\mathbb P_1(D),
$$
which implies $\sym\boldsymbol N=\boldsymbol0$. Therefore $\boldsymbol N\in\mathbb K$ and $\boldsymbol q\in\boldsymbol{RM}$.
\end{proof}   

\begin{lemma}
Let integer $k\geq0$. We have
\begin{equation}\label{eq:xdefimage}
\left(\defm\mathbb P_{k+1}(D;\mathbb R^d)\right)\boldsymbol x=\mathbb P_{k}(D;\mathbb S)\boldsymbol x=\mathbb P_{k+1}(D;\mathbb R^d)\cap\ker(\boldsymbol \pi_{RM}).
\end{equation}
\end{lemma}
\begin{proof}
For any $\boldsymbol\tau\in \mathbb P_{k}(D;\mathbb S)$, it follows
$$
\boldsymbol \pi_{RM}(\boldsymbol\tau\boldsymbol x)= (\skw(\nabla(\boldsymbol\tau\boldsymbol x)))(\boldsymbol 0)\boldsymbol x= \skw(\boldsymbol\tau(\boldsymbol 0))\boldsymbol x=\boldsymbol0.
$$
Thus $\mathbb P_{k}(D;\mathbb S)\boldsymbol x\subseteq\mathbb P_{k+1}(D;\mathbb R^d)\cap\ker(\boldsymbol \pi_{RM})$.
On the other hand, we obtain from Lemma~\ref{lem:xdefkernal} that
$$
\dim\left(\left(\defm\mathbb P_{k+1}(D;\mathbb R^d)\right)\boldsymbol x\right)=\dim\mathbb P_{k+1}(D;\mathbb R^d)-\dim\boldsymbol{RM},
$$
which equals to the dimension of $\mathbb P_{k+1}(D;\mathbb R^d)\cap\ker(\boldsymbol \pi_{RM})$. 
Thus \eqref{eq:xdefimage} follows.
\end{proof}

We denote by $\cdot\boldsymbol x: \mathbb P_{k}(D;\mathbb S) \to \revision{\mathbb P_{k+1}(D;\mathbb R^d)}$ the mapping $\boldsymbol \tau \to \boldsymbol \tau \boldsymbol x$ as the matrix-vector product  $\boldsymbol \tau \boldsymbol x$ is applying row-wise inner product with vector $\boldsymbol x$.
\begin{corollary}
Let integer $k\geq0$. We have the space decomposition 
\begin{equation}\label{eq:symTensorPolySpacedecomp}
\mathbb P_{k}(D;\mathbb S)=\defm \mathbb P_{k+1}(D;\mathbb R^d) \oplus (\ker (\cdot\boldsymbol x)\cap \mathbb P_{k}(D;\mathbb S)).
\end{equation}
\end{corollary}
\begin{proof}
It follows from Lemma~\ref{lem:xdefkernal} that $\defm \mathbb P_{k+1}(D;\mathbb R^d) \cap (\ker (\cdot\boldsymbol x)\cap \mathbb P_{k}(D;\mathbb S))=\{\boldsymbol0\}$.
Due to \eqref{eq:xdefimage},
\begin{align*}
&\quad\dim\defm \mathbb P_{k+1}(D;\mathbb R^d) +\dim(\ker (\cdot\boldsymbol x)\cap \mathbb P_{k}(D;\mathbb S)) \\
&=\dim\defm \mathbb P_{k+1}(D;\mathbb R^d)+\dim\mathbb P_{k}(D;\mathbb S)-\dim(\mathbb P_{k}(D;\mathbb S)\boldsymbol x) \\
&=\dim\mathbb P_{k+1}(D;\mathbb R^d)-\dim\boldsymbol{RM}+\dim\mathbb P_{k}(D;\mathbb S)-\dim(\mathbb P_{k}(D;\mathbb S)\boldsymbol x) \\
&=\dim\mathbb P_{k}(D;\mathbb S),
\end{align*}
which means \eqref{eq:symTensorPolySpacedecomp}.
\end{proof}

\begin{remark}\rm 
In two and three dimensions, we have (cf. \cite{ChenHuang2020,Chen;Huang:2021Finite})
\begin{equation*}
\ker (\cdot\boldsymbol x)\cap \mathbb P_{k}(D;\mathbb S)=\begin{cases}
\boldsymbol x^{\perp}(\boldsymbol x^{\perp})^{\intercal}\mathbb P_{k-2}(D), & \textrm{ for }  d=2,\\
\boldsymbol x\times\mathbb P_{k-2}(D;\mathbb S)\times\boldsymbol x, & \textrm{ for } d=3,
\end{cases}
\end{equation*}
where $\boldsymbol x^{\perp}:=\begin{pmatrix}
x_2\\ -x_1
\end{pmatrix}$, but generalization to arbitrary dimension is not easy and not necessary.
A computation approach to find an explicit basis of $\ker (\cdot\boldsymbol x)\cap \mathbb P_{k-1}(K;\mathbb S)$ is as follows. Find a basis for $\mathbb P_{k-1}(K;\mathbb S)$ and one for $\mathbb P_{k}(K;\mathbb R^d)$. Then form the matrix representation $X$ of the operator $\cdot\boldsymbol x$. Afterwards the null space $\ker(X)$ can be found algebraically. 
\end{remark}



\begin{lemma}\label{lem:E0S}
Let integer $k\geq2$. We have
\begin{equation}\label{eq:E0Sdualcharac} 
E_{0,k}'(\mathbb S) = \mathcal N(\ker (\cdot\boldsymbol x)\cap \mathbb P_{k-2}(K;\mathbb S)).
\end{equation}
That is a function $\boldsymbol\tau\in E_{0,k}(\mathbb S)$ is uniquely determined by 
\begin{equation*}
(\boldsymbol\tau, \boldsymbol q)_K\quad\forall~\boldsymbol q\in \ker (\cdot\boldsymbol x)\cap \mathbb P_{k-2}(K;\mathbb S).
\end{equation*}  
And
$$
\dim E_{0,k}(\mathbb S)= \frac{d(d+1)}{2} {k-2+d \choose d} - d {d+k-1 \choose d} + \frac{d(d+1)}{2}.
$$
\end{lemma}
\begin{proof}

\revision{We apply Lemma \ref{lem:abstract} with $V = \mathbb P_{k}(K;\mathbb S), U= \mathbb P_{k-2}(K;\mathbb S)$, and $\kappa = \cdot\bs x$. Lemma \ref{lem:HuZhanginteriordof} and \eqref{eq:xdefimage}-\eqref{eq:symTensorPolySpacedecomp} verify the assumptions (B1)-(B2), and we only need to count the dimension.}

By the space decomposition \eqref{eq:symTensorPolySpacedecomp}, Lemma~\ref{lem:HuZhanginteriordof} and Lemma~\ref{lem:E0Sbot},
\begin{align*}
\dim(\ker (\cdot\boldsymbol x)\cap \mathbb P_{k-2}(K;\mathbb S))&=\dim\mathbb P_{k-2}(K;\mathbb S)-\dim\defm \mathbb P_{k-1}(K;\mathbb R^d) \\
&=\dim\mathbb B_k(\div, K; \mathbb S)-\dim E_{0,k}^{\bot}(\mathbb S)=\dim E_{0,k}(\mathbb S),
\end{align*}
as required.
\end{proof}


\subsection{$H(\div; \mathbb S)$-conforming elements}

Combining Lemmas \ref{lem:divSboundarydofs}, \ref{lem:E0Sbot}, \ref{lem:E0S} and space decomposition~\eqref{eq:symTensorPolySpacedecomp} yields the degrees of freedom of $H(\div; \mathbb S)$-conforming elements.
\begin{theorem}[$\mathbb P_k(K;\mathbb S)$-type $H(\div; \mathbb S)$-conforming elements]\label{thm:SBDMunisolvence}
Take the shape function space $V(\mathbb S) = \mathbb P_k(K;\mathbb S)$ with $k\geq d+1$. 
The degrees of freedom
\begin{align}
\boldsymbol \tau (\delta) & \quad\forall~\delta\in \mathcal V(K), \label{HdivSBDMfemdof1}\\
(\boldsymbol  n_i^{\intercal}\boldsymbol \tau\boldsymbol n_j, q)_F & \quad\forall~q\in\mathbb P_{k+r-d-1}(F),  F\in\mathcal F^r(K),\;  \label{HdivSBDMfemdof2}\\
&\quad\quad i,j=1,\cdots, r, \textrm{ and } r=1,\cdots, d-1, \notag\\
(\Pi_F\boldsymbol \tau\boldsymbol n, \boldsymbol q)_F & \quad\forall~\boldsymbol q\in {\rm ND}_{k-2}(F),  F\in\mathcal F^1(K),\label{HdivSBDMfemdof3} \\
(\boldsymbol \tau, \boldsymbol q)_K &\quad \forall~\boldsymbol q\in \mathbb P_{k-2}(K;\mathbb S)
\label{HdivSBDMfemdof4}
\end{align}
are uni-solvent for $\mathbb P_k(K;\mathbb S)$.
 The last degree of freedom \eqref{HdivSBDMfemdof4} can be replaced by
 \begin{align}
 (\div\boldsymbol \tau, \boldsymbol q)_K &\quad~\forall~\boldsymbol q\in   \mathbb P_{k-1}(K;\mathbb R^d)/\boldsymbol{RM}, \label{HdivSdef}\\
 (\boldsymbol \tau, \boldsymbol q)_K &\quad~\forall~\boldsymbol q\in \ker (\cdot\boldsymbol x)\cap \mathbb P_{k-2}(K;\mathbb S) \label{HdivSker}.
 \end{align} 
\end{theorem}
The global finite element space $\boldsymbol V_h(\div;\mathbb S)\subset H(\div, \Omega; \mathbb S)$, where
\begin{align*}
\boldsymbol V_h(\div;\mathbb S):=\{\boldsymbol \tau\in \boldsymbol L^2(\Omega;\mathbb S): &\,\boldsymbol \tau|_K\in\mathbb P_k(K;\mathbb S) \textrm{ for each } K\in\mathcal T_h, \\
&\textrm{the degrees of freedom \eqref{HdivSBDMfemdof1}-\eqref{HdivSBDMfemdof3} are single-valued} \}.    
\end{align*}
Clearly $\boldsymbol V_h(\div;\mathbb S)\subset H(\div, \Omega; \mathbb S)$ follows from the proof of Lemma~\ref{lem:divSboundarydofs}.

For the most important three dimensional case, the degrees of freedom \eqref{HdivSBDMfemdof1}-\eqref{HdivSBDMfemdof4} become
\begin{align*}
\boldsymbol \tau (\delta) & \quad\forall~\delta\in \mathcal V(K), \\
(\boldsymbol  n_i^{\intercal}\boldsymbol \tau\boldsymbol n_j, q)_e & \quad\forall~q\in\mathbb P_{k-2}(e),  e\in\mathcal F^2(K), i,j=1,2,\\
(\boldsymbol  n^{\intercal}\boldsymbol \tau\boldsymbol n, q)_F & \quad\forall~q\in\mathbb P_{k-3}(F),  F\in\mathcal F^1(K),\\
(\Pi_F\boldsymbol \tau\boldsymbol n, \boldsymbol q)_F & \quad\forall~\boldsymbol q\in {\rm ND}_{k-2}(F),  F\in\mathcal F^1(K), \\
(\boldsymbol \tau, \boldsymbol q)_K &\quad \forall~\boldsymbol q\in \mathbb P_{k-2}(K;\mathbb S),
\end{align*}
which are slightly different from Hu-Zhang element in three dimensions~\cite{HuZhang2015}.

\revision
{Uni-solvence holds for $k\geq 1$. The requirement $k\geq d+1$ contains the degrees of freedom $(\boldsymbol \tau\boldsymbol n, \boldsymbol q)_F$ for all $\boldsymbol q\in \mathbb P_{1}(F; \mathbb R^{d-1})$ on each face $F\in\mathcal F^1(K)$, by which the divergence of the global $H(\div; \mathbb S)$-conforming element space will include the piecewise $\boldsymbol{RM}$ space and combining with $\div\mathbb B_k(\div, K; \mathbb S)=\mathbb P_{k-1,{\rm RM}}^{\bot}$ will imply the following discrete inf-sup condition.
\begin{lemma}\label{lm:divSinfsup}
Let $k\geq d+1$. It holds the inf-sup condition
$$
\|\boldsymbol p_h\|_0\lesssim\sup_{\boldsymbol\tau_h\in \boldsymbol V_h(\div;\mathbb S)}\frac{(\div\boldsymbol\tau_h, \boldsymbol p_h)}{\|\boldsymbol\tau_h\|_{H(\div)}}\qquad\forall~\boldsymbol p_h\in\mathbb P_{k-1}(\mathcal T_h; \mathbb{R}^d),
$$
where $\mathbb P_{k-1}(\mathcal T_h; \mathbb{R}^d):=\{\boldsymbol p_h\in\boldsymbol L^2(\Omega;\mathbb{R}^d): \boldsymbol p_h|_K\in\mathbb P_{k-1}(K; \mathbb{R}^d) \textrm{ for each } K\in\mathcal T_h\}$.
\end{lemma}}
\revision{
\begin{proof}
For $\boldsymbol p_h\in\mathbb P_{k-1}(\mathcal T_h; \mathbb{R}^d)$, there exists $\boldsymbol\tau\in\boldsymbol H^1(\Omega;\mathbb S)$ such that \cite{CostabelMcIntosh2010}
$$
\div\boldsymbol\tau=\boldsymbol p_h,\quad \|\boldsymbol\tau\|_1\lesssim \|\boldsymbol p_h\|_0.
$$
Let $\boldsymbol\tau_h\in \boldsymbol V_h(\div;\mathbb S)$ such that all the DoFs \eqref{HdivSBDMfemdof1}-\eqref{HdivSBDMfemdof3} and \eqref{HdivSdef}-\eqref{HdivSker} vanish except
\begin{align*}
(\boldsymbol\tau_h\boldsymbol n, \boldsymbol q)_F&=(\boldsymbol\tau\boldsymbol n, \boldsymbol q)_F\qquad\qquad\qquad\;\;\;\,\forall~\boldsymbol q\in\mathbb P_1(F;\mathbb R^d), F\in\mathcal F^1(K), \\
(\div\boldsymbol\tau_h, \boldsymbol q)_K&=(\div\boldsymbol\tau, \boldsymbol q)_K=(\boldsymbol p_h, \boldsymbol q)_K\quad\forall~\boldsymbol q\in\mathbb P_{k-1}(K;\mathbb R^d)/\boldsymbol{RM}
\end{align*}
for all $K\in\mathcal T_h$. By the scaling argument, we have
\begin{equation}\label{eq:20220120}  
\|\boldsymbol\tau_h\|_{0}\lesssim \|\boldsymbol\tau\|_{1}\lesssim \|\boldsymbol p_h\|_0.
\end{equation}
Applying the integration by parts,
$$
(\div\boldsymbol\tau_h, \boldsymbol q)_K=(\div\boldsymbol\tau, \boldsymbol q)_K=(\boldsymbol p_h, \boldsymbol q)_K\quad\forall~\boldsymbol q\in\boldsymbol{RM}.
$$
Hence
$$
(\div\boldsymbol\tau_h, \boldsymbol q)_K=(\div\boldsymbol\tau, \boldsymbol q)_K=(\boldsymbol p_h, \boldsymbol q)_K\quad\forall~\boldsymbol q\in\mathbb P_{k-1}(K;\mathbb R^d),
$$
which implies $\div\boldsymbol\tau_h=\boldsymbol p_h$. Therefore we derive the inf-sup condition from \eqref{eq:20220120}.
\end{proof}
}

\subsection{\revision{$\mathbb P_{k+1}^{-}(K;\mathbb S)$-type elements}}
\revision{
Let $k\geq d+1$. The space of shape functions is taken as 
\begin{equation*}
\mathbb P_{k+1}^{-}(K;\mathbb S):=\mathbb P_k(K;\mathbb S)+E_{0,k+1}^{\bot}(\mathbb S).
\end{equation*}
Since $E_{0,k+1}^{\bot}(\mathbb S)\subseteq \mathbb B_{k+1}(\div, K; \mathbb S)$ and $\div E_{0,k+1}^{\bot}(\mathbb S)=\mathbb P_{k,{\rm RM}}^{\bot}$, we have
$$
{\rm tr}^{\div}\mathbb P_{k+1}^{-}(K;\mathbb S)={\rm tr}^{\div}( \mathbb P_k(K; \mathbb S)),\quad \div \mathbb P_{k+1}^{-}(K;\mathbb S)=\mathbb P_{k}(K;\mathbb R^d),
$$
By applying Lemma~\ref{lm:VPfem} with $\dd=\div$, $V=\mathbb P_k(K;\mathbb S)$, $\mathbb H= E_{0,k+1}^{\bot}(\mathbb S) \backslash E_{0,k}^{\bot}(\mathbb S)$, $\mathbb P=\mathbb P_{k-1}(K,\mathbb R^d)$ and $\mathbb Q=\ker (\cdot\boldsymbol x)\cap \mathbb P_{k-2}(K;\mathbb S)$,
we get the uni-solvent DoFs
\begin{align}
\boldsymbol \tau (\delta) & \quad\forall~\delta\in \mathcal V(K), \label{HdivSRTfemnewdof1}\\
(\boldsymbol  n_i^{\intercal}\boldsymbol \tau\boldsymbol n_j, q)_F & \quad\forall~q\in\mathbb P_{k+r-d-1}(F),  F\in\mathcal F^r(K),\;  \label{HdivSRTfemnewdof2}\\
&\quad\quad i,j=1,\cdots, r, \textrm{ and } r=1,\cdots, d-1, \notag\\
(\Pi_F\boldsymbol \tau\boldsymbol n, \boldsymbol q)_F & \quad\forall~\boldsymbol q\in {\rm ND}_{k-2}(F),  F\in\mathcal F^1(K),\label{HdivSRTfemnewdof3} \\
(\boldsymbol \tau, \boldsymbol q)_K &\quad \forall~\boldsymbol q\in \ker (\cdot\boldsymbol x)\cap \mathbb P_{k-2}(K;\mathbb S), \label{HdivSRTfemnewdof4} \\
(\div\boldsymbol \tau, \boldsymbol q)_K &\quad \forall~\boldsymbol q\in \mathbb P_{k}(K;\mathbb R^d)/\boldsymbol{RM}. \label{HdivSRTfemnewdof5}
\end{align}
}
Since $\div \mathbb P_{k+1}^{-}(K;\mathbb S)=\mathbb P_{k}(K;\mathbb R^d)$ and $\div \mathbb P_k(K;\mathbb S)=\mathbb P_{k-1}(K;\mathbb R^d)$, it is expected that using $\mathbb P_{k+1}^{-}(K;\mathbb S)$ to discretize the mixed elasticity problem will possess one-order higher convergence rate of \revision{the divergence of the discrete stress} than that of $\mathbb P_k(K;\mathbb S)$ symmetric element.

\revision{
\begin{remark}\rm
By the DoFs \eqref{HdivSBDMfemdof1}-\eqref{HdivSBDMfemdof4}, we can find a basis $\{\phi_i\}_{i=1}^{N_1}$ of the bubble function space $\mathbb B_k(\div, K;\mathbb S)$. Let $\{\psi_i\}_{i=1}^{N_2}$ be a basis of $\mathbb P_{k-1}(K;\mathbb R^d)\backslash \boldsymbol{RM}$. Then form the matrix $\big((\div \phi_i,\psi_j)_K\big)_{N_1\times N_2}$, whose kernel space combined with $\{\phi_i\}_{i=1}^{N_1}$ yields the basis of $E_{0,k}(\mathbb S)$. Finally, a basis of $E_{0,k}^{\bot}(\mathbb S)$ is achieved by finding the orthogonal complement of the basis of $E_{0,k}(\mathbb S)$ under the inner product $(\cdot, \cdot)_K$. 
\end{remark}
}

The global finite element space $\boldsymbol V_h^-(\div;\mathbb S)\subset H(\div, \Omega; \mathbb S)$, where
\begin{align*}
\boldsymbol V_h^-(\div;\mathbb S):=\{\boldsymbol \tau\in \boldsymbol L^2(\Omega;\mathbb S): &\,\boldsymbol \tau|_K\in\mathbb P_{k+1}^{-}(K;\mathbb S) \textrm{ for each } K\in\mathcal T_h, \\
&\textrm{the degrees of freedom \eqref{HdivSRTfemnewdof1}-\eqref{HdivSRTfemnewdof3} are single-valued} \}.    
\end{align*}

\revision
{Similarly as Lemma~\ref{lm:divSinfsup}, we have the following inf-sup condition.
\begin{lemma}
Let $k\geq d+1$. It holds the inf-sup condition
$$
\|\boldsymbol p_h\|_0\lesssim\sup_{\boldsymbol\tau_h\in \boldsymbol V_h^-(\div;\mathbb S)}\frac{(\div\boldsymbol\tau_h, \boldsymbol p_h)}{\|\boldsymbol\tau_h\|_{H(\div)}}\qquad\forall~\boldsymbol p_h\in\mathbb P_{k}(\mathcal T_h; \mathbb{R}^d).
$$
\end{lemma}
As RT element, it is natural to enrich $\mathbb P_k(K;\mathbb S)$ to $\mathbb P_k(K;\mathbb S)\oplus\sym(\mathbb H_{k}(K;\mathbb R^d)\boldsymbol x^{\intercal})$. Unfortunately, ${\rm tr}^{\div}(\sym(\mathbb H_{k}(K;\mathbb R^d)\boldsymbol x^{\intercal}))\not\subseteq {\rm tr}^{\div}( \mathbb P_k(K; \mathbb S))$, i.e. assumption (H2) in Lemma~\ref{lm:VPfem} does not hold, which ruins the discrete inf-sup condition. 
}

\section{Symmetric DivDiv-Conforming Finite Elements}\label{sec:divdivS}

We use the previous building blocks to construct divdiv-conforming finite elements in arbitrary dimension. Motivated by the recent construction \cite{Hu;Ma;Zhang:2020family} in two and three dimensions, we first construct $H(\div\div)\cap H(\div)$-conforming finite elements for symmetric tensors and then apply a simple modification to construct $H(\div\div)$-conforming finite elements. We also extend the construction to obtain a new family of $\mathbb P_{k+1}^{-}(\mathbb S)$-type elements.

\subsection{Divdiv operator and Green's identity}
\begin{lemma}\label{lem:divdivonto}
For integer $k\geq 1$, the operator
\begin{equation*}
\div\div:  \boldsymbol x\boldsymbol x^{\intercal}\mathbb H_{k-1}(D) \to \mathbb H_{k-1}(D)
\end{equation*}
is bijective. Consequently $\div\div: \mathbb P_{k+1}(D;\mathbb S)\to \mathbb P_{k-1}(D)$ is surjective. 
\end{lemma}
\begin{proof}
By \eqref{eq:Hkdiv}, it follows
$$
\div\div(\boldsymbol x\boldsymbol x^{\intercal}q)=\div((k+d)\boldsymbol x q)=(k+d)(k+d-1)q\quad\forall~q\in\mathbb H_{k-1}(D),
$$
which ends the proof.
\end{proof}

Next recall the Green's identity for operator divdiv in~\cite{Chen;Huang:2020Finite}. 
\begin{lemma}
We have for any $\boldsymbol \tau\in \mathcal C^2(K; \mathbb S)$ and $v\in H^2(K)$ that
\begin{align}
(\div\div\boldsymbol \tau, v)_K&=(\boldsymbol \tau, \nabla^2v)_K -\sum_{F\in\mathcal F^{1}(K)}\sum_{e\in\mathcal F^1(F)}(\boldsymbol n_{F,e}^{\intercal}\boldsymbol \tau \boldsymbol n, v)_e \notag\\
&\quad - \sum_{F\in\mathcal F^{1}(K)}\left[(\boldsymbol  n^{\intercal}\boldsymbol \tau\boldsymbol  n, \partial_n v)_{F} -  ( \boldsymbol n^{\intercal}\div \boldsymbol \tau +  \div_F(\boldsymbol\tau \boldsymbol n), v)_F\right]. \label{eq:greenidentitydivdiv}
\end{align}
\end{lemma}
\begin{proof}
We start from the standard integration by parts
\begin{align*}
\begin{aligned}
(\operatorname{div}\div \boldsymbol  \tau, v)_{K} &=-(\div\boldsymbol \tau, \nabla v)_{K}+\sum_{F \in \mathcal F^1(K)}(\boldsymbol  n^{\intercal}\div  \boldsymbol \tau, v)_F \\
&=\left(\boldsymbol  \tau, \nabla^{2} v\right)_{K}-\sum_{F \in \mathcal F^1(K)} (\boldsymbol  \tau \boldsymbol  n, \nabla v)_F+\sum_{F \in \mathcal F^1(K)}(\boldsymbol  n^{\intercal} \div  \boldsymbol   \tau, v)_F.
\end{aligned}
\end{align*}
We then decompose $\nabla v = \partial_nv\boldsymbol n+ \nabla_F v$ and apply the Stokes theorem to get
\begin{align*}
(\boldsymbol  \tau \boldsymbol  n, \nabla v)_F&=(\boldsymbol \tau \boldsymbol n, \partial_nv\boldsymbol n+ \nabla_F v)_F \\
&=(\boldsymbol n^{\intercal}\boldsymbol \tau \boldsymbol n, \partial_nv)_F -( \div_F(\boldsymbol \tau\boldsymbol n),  v)_F + \sum_{e\in\mathcal F^1(F)} (\boldsymbol n_{F,e}^{\intercal}\boldsymbol \tau \boldsymbol n, v)_e.
\end{align*}
Thus the Green's identity~\eqref{eq:greenidentitydivdiv} follows from the last two identities.
\end{proof}

\subsection{$H(\div\div; \mathbb S)\cap H(\div; \mathbb S)$-conforming elements}

Based on \eqref{eq:greenidentitydivdiv}, it suffices to enforce the continuity of both $\boldsymbol \tau \boldsymbol n$ and $\boldsymbol n^{\intercal}\div \boldsymbol \tau$ so that the constructed finite element space is $H(\operatorname{div}, \mathbb{S})\cap H(\operatorname{div}\div, \mathbb{S})$-conforming. Such an approach is recently proposed in \cite{Hu;Ma;Zhang:2020family} to construct two and three dimensional finite elements. The readers are referred to Fig. \ref{fig:divdivcomp} for an illustration of the space decomposition. 

\begin{figure}[htbp]
\begin{center}
\includegraphics[width=10.5cm]{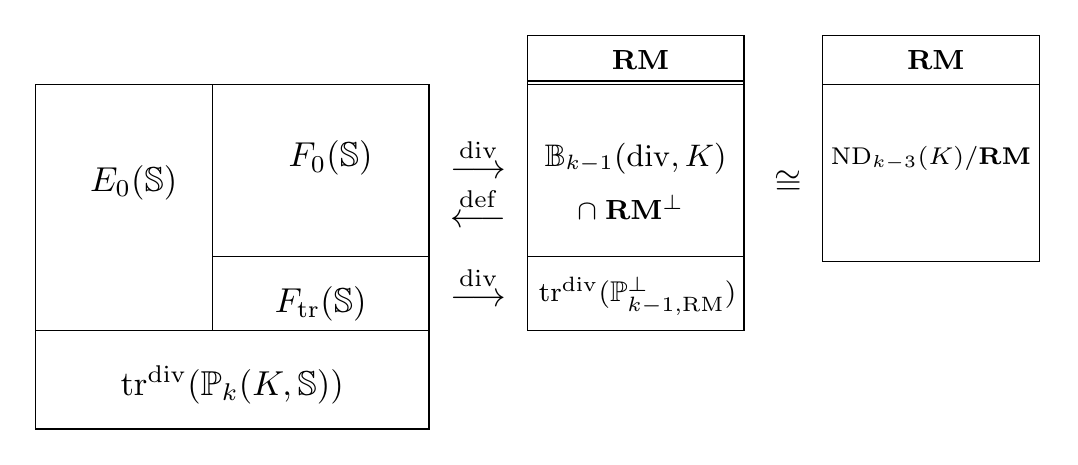}
\caption{Decomposition of $\mathbb P_k(K,\mathbb S)$ for an $H(\div\div)\cap H(\div)$ conforming finite element.}
\label{fig:divdivcomp}
\end{center}
\end{figure}

The subspaces $\tr^{\div}(\mathbb P_k(K,\mathbb S))$ and $E_0(\mathbb S)$ are unchanged. The space $\div E_0^{\bot}(\mathbb S) = \mathbb P_{k-1,{\rm RM}}^{\bot}$ will be further split by the trace operator. Define 
$$
F_0(\mathbb S)\subseteq E_0^{\bot}(\mathbb S), \text{ satisfying } \div F_0(\mathbb S) = \mathbb B_{k-1}(\div, K)\cap \boldsymbol{RM}^{\bot},
$$ 
and
\revision{$F_{\tr}(\mathbb S)\subseteq E_0^{\bot}(\mathbb S)$ with ${\rm tr}^{\div}(\div F_{\tr}(\mathbb S)) = {\rm tr}^{\div}(\div E_0^{\bot}) = {\rm tr}^{\div}(\mathbb P_{k-1,{\rm RM}}^{\bot})$},
which is well defined as $\div$ restricted to $E_0^{\bot}(\mathbb S)$ is a bijection. Here
$$
\mathbb B_{k-1}(\div, K)\cap \boldsymbol{RM}^{\bot}:=\{\boldsymbol v\in\mathbb B_{k-1}(\div, K): (\boldsymbol v, \boldsymbol q)_K=0\;\textrm{ for all } \boldsymbol q\in\boldsymbol{RM}\}.
$$

 \begin{lemma}\label{lm:Ftr}
For integer $k\geq 3$, it holds
 $$
\revision{ {\rm tr}^{\div}(\div F_{\tr}(\mathbb S)) }
   = {\rm tr}^{\div}(\mathbb P_{k-1}(K;\mathbb R^d)).
 $$ 
 \revision{ Consequently $({\rm tr}^{\div}(\div F_{\tr}(\mathbb S)))'=\mathcal N( \mathbb P_{k-1}(\mathcal F^1(K)))$.} 
\end{lemma}
 \begin{proof}
By definition, $ {\rm tr}^{\div}(\div F_{\tr}(\mathbb S)) = {\rm tr}^{\div}(\mathbb P_{k-1,{\rm RM}}^{\bot})\subseteq{\rm tr}^{\div}(\mathbb P_{k-1}(K;\mathbb R^d))$.
On the other hand, given a trace $p\in {\rm tr}^{\div}(\mathbb P_{k-1}(K;\mathbb R^d))$, by the uni-solvence of BDM element, cf. Theorem~\ref{thm:BDMunisolvence},
we can find a $\boldsymbol v\in \mathbb P_{k-1}(K;\mathbb R^d)$ such that $\boldsymbol v\cdot\boldsymbol n = p$ on $\partial K$ and $\boldsymbol v \perp \boldsymbol{RM}$ as $\boldsymbol{RM}={\rm ND}_{0}(K)\subseteq {\rm ND}_{k-3}(K)$ when $k\geq 3$.
\end{proof}

\begin{lemma}\label{lm:F0star}
For integer $k\geq 3$, we have
$$
F_0'(\mathbb S) = \mathcal N(\defm ({\rm ND}_{k-3}(K))).
$$
\end{lemma}
\begin{proof}
We pick a $\boldsymbol \tau \in F_0(\mathbb S)$, i.e., $\boldsymbol \tau$ satisfies
$$
(\boldsymbol \tau\boldsymbol n)|_{\partial K}= 0, \quad \boldsymbol n^{\intercal} \div \boldsymbol \tau |_{\partial K} = 0, \quad \boldsymbol \tau \perp  E_0(\mathbb S).
$$ 
Assume 
$$
(\boldsymbol \tau, \defm \boldsymbol q)_K = 0, \quad \forall~\boldsymbol q\in {\rm ND}_{k-3}(K).
$$
Note that $\boldsymbol v =\div \boldsymbol \tau \in \mathbb B_{k-1}(\div,K)$, and $(\boldsymbol v, \boldsymbol q)_K = 0$ for all $\boldsymbol q\in {\rm ND}_{k-3}(K)$, then $\boldsymbol v = \boldsymbol 0$ by Theorem \ref{thm:BDMunisolvence}. Therefore $\div \boldsymbol \tau = \boldsymbol 0$, i.e., $\boldsymbol \tau\in E_0(\mathbb S)$. As $\boldsymbol \tau \perp  E_0(\mathbb S)$, the only possibility is $\boldsymbol \tau = \boldsymbol 0$. 

\revision{Then the dimension count 
$$
\dim F_0(\mathbb S) = \dim \mathbb B_{k-1}(\div, K) - \dim \boldsymbol{RM} = \dim {\rm ND}_{k-3}(K) - \dim \ker(\defm)
$$
will finish the proof.}
\end{proof}

We summarize the construction in the following theorem.
\begin{theorem}\label{th:divdivcapdiv}
 Let $V(\div\div^+;\mathbb S) := \mathbb P_k(K,\mathbb S)$ with $k\geq \max\{d,3\}$. Then the following set of degrees of freedom determines an $H(\div\div; \mathbb S)\cap H(\div; \mathbb S)$-conforming finite element
\begin{align}
\boldsymbol \tau (\delta) & \quad\forall~\delta\in \mathcal V(K), \label{HdivdivSfemdof1}\\
(\boldsymbol  n_i^{\intercal}\boldsymbol \tau\boldsymbol n_j, q)_F & \quad\forall~q\in\mathbb P_{k+r-d-1}(F),  F\in\mathcal F^r(K),\;  \label{HdivdivSfemdof2}\\
&\quad\quad i,j=1,\cdots, r, \textrm{ and } r=1,\cdots, d-1, \notag\\
(\Pi_F\boldsymbol \tau\boldsymbol n, \boldsymbol q)_F & \quad\forall~\boldsymbol q\in {\rm ND}_{k-2}(F),  F\in\mathcal F^{1}(K),\label{HdivdivSfemdof3}\\
(\boldsymbol n^{\intercal}\div\boldsymbol\tau, q)_F &\quad\forall~q\in\mathbb P_{k-1}(F), F\in \mathcal F^{1}(K),\label{HdivdivSfemdof4}\\
(\div\div\boldsymbol \tau, q)_K &\quad \forall~q\in \mathbb P_{k-2}(K)/\mathbb P_{1}(K),\label{HdivdivSfemdof51}\\
(\div\boldsymbol \tau, \boldsymbol q)_K &\quad \forall~\boldsymbol q\in \big(\mathbb P_{k-3}(K; \mathbb K)/\mathbb P_{0}(K; \mathbb K)\big)\boldsymbol x,\label{HdivdivSfemdof52}\\
(\boldsymbol \tau, \boldsymbol q)_K &\quad \forall~\boldsymbol q\in  \ker (\cdot\boldsymbol x)\cap \mathbb P_{k-2}(K;\mathbb S) \label{HdivdivSfemdof6}.
\end{align}
\end{theorem}
\begin{proof}
By Lemma~\ref{lem:divSboundarydofs}, the vanishing degrees of freedom \eqref{HdivdivSfemdof1}-\eqref{HdivdivSfemdof3} imply $\boldsymbol \tau \boldsymbol n |_{\partial K}= \boldsymbol0$. Then applying Lemmas~\ref{lm:Ftr}-\ref{lm:F0star},
we get from the vanishing degrees of freedom \eqref{HdivdivSfemdof4}-\eqref{HdivdivSfemdof52} that $\boldsymbol \tau \in E_0(\mathbb S)$.
Finally combining \eqref{eq:E0Sdualcharac} and \eqref{HdivdivSfemdof6} implies $\boldsymbol \tau = \boldsymbol0$. 

We then count the dimensions. Compared to the degrees of freedom of BDM-type $H(\div,\mathbb S)$ element, cf. Theorem \ref{thm:SBDMunisolvence}, the difference is \eqref{HdivSdef} v.s. \eqref{HdivdivSfemdof4}-\eqref{HdivdivSfemdof52}. Then from the uni-solvence of BDM div-conforming element, cf. Theorem \ref{thm:BDMunisolvence}, we have
$$
\dim \mathbb P_{k-1}(K; \mathbb R^d) = \dim {\rm ND}_{k-3}(K) + \sum_{F\in \mathcal F^{1}(K)}\dim\mathbb P_{k-1}(F), 
$$
and consequently the number of degrees of freedom \eqref{HdivdivSfemdof1}-\eqref{HdivdivSfemdof6} is $\dim\mathbb P_k(K;\mathbb S)$. 
\end{proof}
The global finite element space $\boldsymbol V_h(\div\div^+;\mathbb S)\subset H(\div\div, \Omega; \mathbb S)\cap H(\div, \Omega; \mathbb S)$ is defined as follows
\begin{align*}
\boldsymbol V_h(\div\div^+;\mathbb S):=\{&\boldsymbol \tau\in \boldsymbol L^2(\Omega;\mathbb S): \boldsymbol \tau|_K\in \mathbb P_k(K,\mathbb S) \textrm{ for each } K\in\mathcal T_h, \\
&\textrm{ the degrees of freedom \eqref{HdivdivSfemdof1}-\eqref{HdivdivSfemdof4} are single-valued} \}.    
\end{align*}

The requirement $k\geq d$ ensures the degrees of freedom $(\boldsymbol n^{\intercal}\boldsymbol \tau\boldsymbol n, q)_F$ for all $q\in \mathbb P_{0}(F)$
 on each face $F\in\mathcal F^1(K)$, by which
space $\div\div\boldsymbol V_h(\div\div^+;\mathbb S)$ will include all the piecewise linear functions.

\revision{\begin{lemma}\label{lm:divdivdivSinfsup}
Let $k\geq \max\{d,3\}$.
It holds the inf-sup condition
$$
\|p_h\|_0\lesssim\sup_{\boldsymbol\tau_h\in \boldsymbol V_h(\div\div^+;\mathbb S)}
\frac{(\div\div\boldsymbol{\tau}_h, p_h)}{\|\boldsymbol\tau_h\|_{H(\div)}+
\|\div\div\boldsymbol{\tau}_h\|_{0}}\qquad\forall~p_h\in\mathbb P_{k-2}(\mathcal T_h),
$$
where $\mathbb P_{k-2}(\mathcal T_h):=\{p_h\in L^2(\Omega): p_h|_K\in\mathbb P_{k-2}(K) \textrm{ for each } K\in\mathcal T_h\}$.
\end{lemma}
}\begin{proof}
For $p_h\in\mathbb P_{k-2}(\mathcal T_h)$, there exists $\boldsymbol\tau\in\boldsymbol H^2(\Omega;\mathbb S)$ such that \cite{CostabelMcIntosh2010}
$$
\div\boldsymbol\tau=p_h,\quad \|\boldsymbol\tau\|_2\lesssim \|p_h\|_0.
$$
Let $\boldsymbol\tau_h\in \boldsymbol V_h(\div\div^+;\mathbb S)$ such that all the DoFs \eqref{HdivdivSfemdof1}-\eqref{HdivdivSfemdof6} vanish except
\begin{align*}
(\boldsymbol\tau_h\boldsymbol n, \boldsymbol q)_F&=(\boldsymbol\tau\boldsymbol n, \boldsymbol q)_F\qquad\qquad\qquad\quad\;\;\;\;\forall~\boldsymbol q\in\mathbb P_0(F;\mathbb R^d), F\in\mathcal F^1(K), \\
(\boldsymbol n^{\intercal}\div\boldsymbol\tau_h, q)_F&=(\boldsymbol n^{\intercal}\div\boldsymbol\tau, q)_F\qquad\qquad\quad\;\;\;\,\forall~q\in\mathbb P_1(F;\mathbb R^d), F\in\mathcal F^1(K), \\
(\div\div\boldsymbol\tau_h, q)_K&=(\div\div\boldsymbol\tau, q)_K=(p_h, q)_K\quad\forall~q\in \mathbb P_{k-2}(K)/\mathbb P_{1}(K)
\end{align*}
for all $K\in\mathcal T_h$. By the scaling argument, we have
\begin{equation}\label{eq:202201201}  
\|\boldsymbol\tau_h\|_{H(\div)}\lesssim \|\boldsymbol\tau\|_{2}\lesssim \|p_h\|_0.
\end{equation}
Applying the integration by parts,
$$
(\div\div\boldsymbol\tau_h, q)_K=(\div\div\boldsymbol\tau, q)_K=(p_h, q)_K\quad\forall~q\in\mathbb P_1(K).
$$
Hence
$$
(\div\div\boldsymbol\tau_h, q)_K=(\div\div\boldsymbol\tau, q)_K=(p_h, q)_K\quad\forall~q\in\mathbb P_{k-2}(K),
$$
which implies $\div\div\boldsymbol\tau_h=p_h$. Therefore we derive the inf-sup condition from \eqref{eq:202201201}.
\end{proof}



\subsection{$\mathbb P_{k+1}^{-}(\mathbb S)$-type $H(\div\div; \mathbb S)\cap H(\div; \mathbb S)$-conforming elements}

The space of shape functions is taken as 
\begin{equation*}
V^{-}(\div\div^+;\mathbb S):=\mathbb P_k(K;\mathbb S) \oplus \boldsymbol x\boldsymbol x^{\intercal}\mathbb H_{k-1}(K)
\end{equation*}
with $k\geq \max\{d,3\}$.
The degrees of freedom are
\begin{align}
\boldsymbol \tau (\delta) & \quad\forall~\delta\in \mathcal V(K), \label{HdivdivSRTfemdof1}\\
(\boldsymbol  n_i^{\intercal}\boldsymbol \tau\boldsymbol n_j, q)_F & \quad\forall~q\in\mathbb P_{k+r-d-1}(F),  F\in\mathcal F^r(K),\;  \label{HdivdivSRTfemdof2}\\
&\quad\quad i,j=1,\cdots, r, \textrm{ and } r=1,\cdots, d-1, \notag\\
(\Pi_F\boldsymbol \tau\boldsymbol n, \boldsymbol q)_F & \quad\forall~\boldsymbol q\in {\rm ND}_{k-2}(F),  F\in\mathcal F^{1}(K),\label{HdivdivSRTfemdof3}\\
(\boldsymbol  n^{\intercal}\div \boldsymbol \tau, p)_F &\quad\forall~p\in\mathbb P_{k-1}(F), F\in \mathcal F^{1}(K),\label{HdivdivSRTfemdof4}\\
(\div\div\boldsymbol \tau, q)_K &\quad \forall~q\in \mathbb P_{k-1}(K)/\mathbb P_{1}(K),\label{HdivdivSRTfemdof51}\\
(\div\boldsymbol \tau, \boldsymbol q)_K &\quad \forall~\boldsymbol q\in \big(\mathbb P_{k-3}(K; \mathbb K)/\mathbb P_{0}(K; \mathbb K)\big)\boldsymbol x,\label{HdivdivSRTfemdof52}\\
(\boldsymbol \tau, \boldsymbol q)_K &\quad \forall~\boldsymbol q\in  \ker (\cdot\boldsymbol x)\cap \mathbb P_{k-2}(K;\mathbb S) \label{HdivdivSRTfemdof6}.
\end{align}

\revision{We can see that $\mathbb P_{k+1}^{-}(\mathbb S)$-type $H(\div\div; \mathbb S)\cap H(\div; \mathbb S)$-conforming elements follows from Lemma~\ref{lm:VPfem} with $\dd=\div\div$, $V=\mathbb P_k(K;\mathbb S)$, $\mathbb H=\boldsymbol x\boldsymbol x^{\intercal}\mathbb H_{k-1}(K)$, $\mathbb P=\mathbb P_{k-1}(K)/\mathbb P_{1}(K)$ and $\mathbb Q=\ker (\cdot\boldsymbol x)\cap \mathbb P_{k-2}(K;\mathbb S)$. The assumption (H5) holds from the fact $\div\div\mathbb B^+=\mathbb P_{k-1}(K)/\mathbb P_{1}(K)$ and $\nabla^2(\mathbb P+\dd\mathbb H)=\nabla^2\mathbb P_{k-1}(K)$.}

Due to the added component $\boldsymbol x\boldsymbol x^{\intercal}\mathbb H_{k-1}(K)$, the range of $\div\div$ operator is increased to $\mathbb P_{k-1}(K)$ instead of $\mathbb P_{k-2}(K)$. The degree of freedom $(\div\boldsymbol \tau, \boldsymbol q)_K $ is increased from $\boldsymbol q \in  {\rm ND}_{k-3}(K)=\grad \mathbb P_{k-2}(K)\oplus\mathbb P_{k-3}(K;\mathbb K)\boldsymbol x$ to $\mathbb P_{k-2}(K;\mathbb R^d)=\grad \mathbb P_{k-1}(K)\oplus\mathbb P_{k-3}(K;\mathbb K)\boldsymbol x$. 
Hence the number of degrees of freedom \eqref{HdivdivSRTfemdof1}-\eqref{HdivdivSRTfemdof6} equals to $\dim V^{-}(\div\div^+;\mathbb S)$.
The boundary DoFs, however, remains the same as $(\boldsymbol x\boldsymbol x^{\intercal}\mathbb H_{k-1}(K))\boldsymbol n|_F\in \mathbb P_{k}(F;\mathbb R^d)$. 
It is expected that using the $\mathbb P_{k+1}^{-}(\mathbb S)$-type symmetric element to discretize the biharmonic problem will possess one-order higher convergence rate of \revision{the $\div\div$ of the discrete bending moment} than that of the $\mathbb P_k(\mathbb S)$-type symmetric element while the computational cost is not increased significantly; \revision{see~\cite[Section 4]{ChenHuang2020}}. When solving the linear algebraic equation, all interior degrees of freedom can be eliminated element-wisely.  




\begin{lemma}\label{lem:E0SRT}
Let $\boldsymbol\tau\in V^{-}(\div\div^+;\mathbb S)$.
If the degrees of freedom \eqref{HdivdivSRTfemdof1}-\eqref{HdivdivSRTfemdof52} vanish, then $\boldsymbol \tau \in E_0(\mathbb S)$.
\end{lemma}
\begin{proof}
Since $\boldsymbol x\cdot \boldsymbol n$ is constant on each $(d-1)$-dimensional face, 
 the trace $\boldsymbol\tau\boldsymbol n|_F\in\mathbb P_k(F;\mathbb R^{d})$ and $(\boldsymbol  n^{\intercal}\div \boldsymbol \tau)|_F\in\mathbb P_{k-1}(F)$ remain unchanged.
Then we conclude  ${\rm tr}^{\rm div}\boldsymbol\tau=\boldsymbol0$ and ${\rm tr}^{\rm div}(\div \boldsymbol \tau)=0$ from Theorem~\ref{th:divdivcapdiv}.

 Applying the Green's identity \eqref{eq:greenidentitydivdiv}, we get 
$$
(\div\div\boldsymbol\tau , v)_K=(\boldsymbol\tau, \nabla^2 v)_K=0 \quad \forall~v \in \mathbb P_{1}(K).
$$
Hence it follows from the vanishing degrees of freedom~\eqref{HdivdivSRTfemdof51} that $\div\div\boldsymbol\tau=0$, which combined with Lemma~\ref{lem:divdivonto} implies $\boldsymbol\tau\in\mathbb P_k(K;\mathbb S)$.
Finally we achieve from Lemma~\ref{lm:F0star} and the vanishing degrees of freedom~\eqref{HdivdivSRTfemdof52} that $\boldsymbol \tau \in E_0(\mathbb S)$.
\end{proof}



Combining Lemma~\ref{lem:E0SRT}, \eqref{eq:E0Sdualcharac} and the degree of freedom \eqref{HdivdivSRTfemdof6} shows the uni-solvence of the $\mathbb P_{k+1}^{-}(\mathbb S)$-type $H(\div\div; \mathbb S)\cap H(\div; \mathbb S)$-conforming elements.
\begin{theorem}\label{th:P-S}
The degrees of freedom \eqref{HdivdivSRTfemdof1}-\eqref{HdivdivSRTfemdof6} are 
uni-solvent for the space $V^{-}(\div\div^+;\mathbb S) = \mathbb P_k(K;\mathbb S) \oplus \boldsymbol x\boldsymbol x^{\intercal}\mathbb H_{k-1}(K)$.
\end{theorem}

The finite element space $\boldsymbol V_h^-(\div\div^+;\mathbb S)\subset H(\div\div, \Omega; \mathbb S)\cap H(\div, \Omega; \mathbb S)$ is then defined as follows
\begin{align*}
\boldsymbol V_h^-(\div\div^+;\mathbb S)&:=\{\boldsymbol \tau\in \boldsymbol L^2(\Omega;\mathbb S): \boldsymbol \tau|_K\in \mathbb P_k(K,\mathbb S)\oplus \boldsymbol x\boldsymbol x^{\intercal}\mathbb H_{k-1}(K) \textrm{ for each } \\
&\; K\in\mathcal T_h, \textrm{ the degrees of freedom \eqref{HdivdivSRTfemdof1}-\eqref{HdivdivSRTfemdof4} are single-valued} \}.    
\end{align*}



\revision{Similarly as Lemma~\ref{lm:divdivdivSinfsup}, we have the following inf-sup condition.
\begin{lemma}\label{lm:divdivdivSminfsup}
Let $k\geq \max\{d,3\}$.
It holds 
$$
\|p_h\|_0\lesssim\sup_{\boldsymbol\tau_h\in \boldsymbol V_h^-(\div\div^+;\mathbb S)}
\frac{(\div\div\boldsymbol{\tau}_h, p_h)}{\|\boldsymbol\tau_h\|_{H(\div)}+
\|\div\div\boldsymbol{\tau}_h\|_{0}}\qquad\forall~p_h\in\mathbb P_{k-1}(\mathcal T_h).
$$
\end{lemma}
}
\subsection{divdiv conforming elements}\label{subsec:newdivdivSfem}
The requirement both $\boldsymbol \tau \boldsymbol n$ and $\boldsymbol n^{\intercal}\div \boldsymbol \tau $ are continuous is sufficient but not necessary for a function to be in $H(\div\div)$.
In addition to $\boldsymbol n^{\intercal}\boldsymbol\tau\boldsymbol n$, the combination $\boldsymbol n^{\intercal}\div \boldsymbol \tau +  \div_F(\boldsymbol\tau \boldsymbol n)$ to be continuous is enough
due to the Green's identity \eqref{eq:greenidentitydivdiv}. 



\begin{theorem}\label{thm:newHdivdivSfem}
Take $V(\div\div;\mathbb S):=\mathbb P_k(K;\mathbb S)$ with $k\geq\max\{d,3\}$, as the space of shape functions.
The degrees of freedom are given by
\begin{align}
\boldsymbol \tau (\delta) & \quad\forall~\delta\in \mathcal V(K), \label{newHdivdivSfemdof1}\\
(\boldsymbol  n_i^{\intercal}\boldsymbol \tau\boldsymbol n_j, q)_F & \quad\forall~q\in\mathbb P_{k+r-d-1}(F),  F\in\mathcal F^r(K),\;  \label{newHdivdivSfemdof2}\\
&\quad\quad i,j=1,\cdots, r, \textrm{ and } r=1,\cdots, d-1, \notag\\
(\Pi_F\boldsymbol \tau\boldsymbol n, \boldsymbol q)_F & \quad\forall~\boldsymbol q\in {\rm ND}_{k-2}(F),  F\in\mathcal F^{1}(K),\label{newHdivdivSfemdof4}\\
(\boldsymbol n^{\intercal}\div \boldsymbol \tau +  \div_F(\boldsymbol\tau \boldsymbol n), p)_F &\quad\forall~p\in\mathbb P_{k-1}(F), F\in \mathcal F^{1}(K),\label{newHdivdivSfemdof3}\\
(\boldsymbol \tau, \defm\boldsymbol q)_K &\quad \forall~\boldsymbol q\in {\rm ND}_{k-3}(K),\label{newHdivdivSfemdof5}\\
(\boldsymbol \tau, \boldsymbol q)_K &\quad \forall~\boldsymbol q\in  \ker (\cdot\boldsymbol x)\cap \mathbb P_{k-2}(K;\mathbb S) \label{newHdivdivSfemdof6}.
\end{align}
The degree of freedom \eqref{newHdivdivSfemdof4} is considered as interior to $K$, i.e., it is not single-valued across elements. 
\end{theorem}
\begin{proof}
By Lemma \ref{lem:divSboundarydofs}, the $\div_F(\boldsymbol \tau\boldsymbol n)$ can be determined by \revision{\eqref{newHdivdivSfemdof1}, \eqref{newHdivdivSfemdof2}}, and \eqref{newHdivdivSfemdof4}. A linear combination with \eqref{newHdivdivSfemdof3}, the trace $\boldsymbol n^{\intercal}\div \boldsymbol \tau$ can be determined. Then the uni-solvence is obtained from Theorem \ref{th:divdivcapdiv}. 
%
\end{proof}



The finite element space $\boldsymbol V_h(\div\div)$ is defined as follows
\begin{align*}
\boldsymbol V_h(\div\div,\Omega;\mathbb S):=\{&\boldsymbol \tau\in \boldsymbol L^2(\Omega;\mathbb S): \boldsymbol \tau|_K\in\mathbb P_k(K;\mathbb S) \textrm{ for each } K\in\mathcal T_h, \textrm{ the} \\
&\textrm{ degrees of freedom \eqref{newHdivdivSfemdof1}-\eqref{newHdivdivSfemdof2} and \eqref{newHdivdivSfemdof3} are single-valued} \}.    
\end{align*}
As $\boldsymbol n^{\intercal}\boldsymbol\tau\boldsymbol n$ and $\boldsymbol n^{\intercal}\div \boldsymbol \tau +  \div_F(\boldsymbol\tau \boldsymbol n)$ are continuous, $\boldsymbol V_h(\div\div)\subset H(\div\div, \Omega; \mathbb S)$; see~\cite[Lemma 4.4]{Chen;Huang:2020Finite}.

Finally we present a $\mathbb P_{k+1}^{-}(\mathbb S)$-type $H(\div\div; \mathbb S)$-conforming element. 

\begin{theorem}
Let integer $k\geq \max\{d,3\}$.
Take the space of shape functions as $$V^{-}(\div\div;\mathbb S):=\mathbb P_k(K;\mathbb S) \oplus \boldsymbol x\boldsymbol x^{\intercal}\mathbb H_{k-1}(K).$$
The degrees of freedom are
\begin{align}
\boldsymbol \tau (\delta) & \quad\forall~\delta\in \mathcal V(K), \label{newHdivdivSRTfemdof1}\\
(\boldsymbol  n_i^{\intercal}\boldsymbol \tau\boldsymbol n_j, q)_F & \quad\forall~q\in\mathbb P_{k+r-d-1}(F),  F\in\mathcal F^r(K),\;  \label{newHdivdivSRTfemdof2}\\
&\quad\quad i,j=1,\cdots, r, \textrm{ and } r=1,\cdots, d-1, \notag\\
(\Pi_F\boldsymbol \tau\boldsymbol n, \boldsymbol q)_F & \quad\forall~\boldsymbol q\in {\rm ND}_{k-2}(F),  F\in\mathcal F^{1}(K),\label{newHdivdivSRTfemdof4}\\
(\boldsymbol n^{\intercal}\div \boldsymbol \tau +  \div_F(\boldsymbol\tau \boldsymbol n), p)_F &\quad\forall~p\in\mathbb P_{k-1}(F), F\in \mathcal F^{1}(K),\label{newHdivdivSRTfemdof3}\\
(\boldsymbol \tau, \defm\boldsymbol q)_K &\quad \forall~\boldsymbol q\in \mathbb P_{k-2}(K; \mathbb R^d),\label{newHdivdivSRTfemdof5}\\
(\boldsymbol \tau, \boldsymbol q)_K &\quad \forall~\boldsymbol q\in  \ker (\cdot\boldsymbol x)\cap \mathbb P_{k-2}(K;\mathbb S) \label{newHdivdivSRTfemdof6}.
\end{align}
Again the degree of freedom \eqref{newHdivdivSRTfemdof4} is considered as interior to $K$, i.e., it is not single-valued across elements.
\end{theorem}
\begin{proof}
 By Lemma \ref{lem:divSboundarydofs}, the $\div_F(\boldsymbol \tau\boldsymbol n)$ can be determined by \eqref{newHdivdivSRTfemdof1}, \eqref{newHdivdivSRTfemdof2}, and \eqref{newHdivdivSRTfemdof4}. A linear combination with \eqref{newHdivdivSRTfemdof3}, the trace $\boldsymbol n^{\intercal}\div \boldsymbol \tau$ can be determined. Then the uni-solvence is obtained from Theorem \ref{th:P-S}. 
\end{proof}

 The global finite element space $\boldsymbol V_h^{-}(\div\div)\subset H(\div\div, \Omega; \mathbb S)$, where
\begin{align*}
\boldsymbol V_h^{-}(\div\div,\Omega;\mathbb S):=\{&\boldsymbol \tau\in \boldsymbol L^2(\Omega;\mathbb S): \boldsymbol \tau|_K\in V^{-}(\div\div;\mathbb S) \textrm{ for each } K\in\mathcal T_h, \textrm{ the} \\
&\textrm{ degrees of freedom \eqref{newHdivdivSRTfemdof1}-\eqref{newHdivdivSRTfemdof2} and \eqref{newHdivdivSRTfemdof3} are single-valued} \}.    
\end{align*}

\revision{Finally we list inf-sup conditions for divdiv conforming elements.
\begin{lemma}
Let $k\geq \max\{d,3\}$.
We have 
\begin{equation}\label{eq:divdivinfsup1}
\|p_h\|_0\lesssim\sup_{\boldsymbol\tau_h\in \boldsymbol V_h(\div\div,\Omega;\mathbb S)}
\frac{(\div\div\boldsymbol{\tau}_h, p_h)}{\|\boldsymbol\tau_h\|_{0}+
\|\div\div\boldsymbol{\tau}_h\|_{0}}\qquad\forall~p_h\in\mathbb P_{k-2}(\mathcal T_h),
\end{equation}
\begin{equation}\label{eq:divdivinfsup2}
\|p_h\|_0\lesssim\sup_{\boldsymbol\tau_h\in \boldsymbol V_h^{-}(\div\div,\Omega;\mathbb S)}
\frac{(\div\div\boldsymbol{\tau}_h, p_h)}{\|\boldsymbol\tau_h\|_{0}+
\|\div\div\boldsymbol{\tau}_h\|_{0}}\qquad\forall~p_h\in\mathbb P_{k-1}(\mathcal T_h).
\end{equation}
\end{lemma}
\begin{proof}
Since $\|\boldsymbol\tau_h\|_{0}\leq \|\boldsymbol\tau_h\|_{H(\div)}$ and
$\boldsymbol V_h(\div\div^+;\mathbb S)\subseteq \boldsymbol V_h(\div\div,\Omega;\mathbb S)$, the inf-sup condition \eqref{eq:divdivinfsup1} follows from Lemma~\ref{lm:divdivdivSinfsup}. Similarly, the inf-sup condition \eqref{eq:divdivinfsup2} follows from Lemma~\ref{lm:divdivdivSminfsup} and $\boldsymbol V_h^-(\div\div^+;\mathbb S)\subseteq \boldsymbol V_h^-(\div\div,\Omega;\mathbb S)$.
\end{proof}
}



\bibliographystyle{siamplain}
\bibliography{./paper}
\end{document}